\documentclass{amsart}
\usepackage[pdftex]{graphicx}

\usepackage{mathtools}
\usepackage{afterpage}
\usepackage{enumitem}
\usepackage{multirow}

\usepackage{bm}
\usepackage{color}
\usepackage{caption} 
\mathtoolsset{showonlyrefs=true}

\numberwithin{equation}{section}

\newtheorem{theorem}{Theorem}[section]
\newtheorem{lemma}[theorem]{Lemma}

\newtheorem{proposition}[theorem]{Proposition}

\theoremstyle{definition}

\newtheorem{example}[theorem]{Example}

\newtheorem*{acknowledgment}{Acknowledgments}

\theoremstyle{remark}
\newtheorem{remark}[theorem]{Remark}

\numberwithin{equation}{section}

\newcommand{\Z}{\mathbb{Z}}
\newcommand{\R}{\mathbb{R}}

\newcommand{\F}{\mathbb{F}}
\newcommand{\CFK}{{\rm CFK}}
\newcommand{\HFK}{{\rm HFK}}

\renewcommand{\u}{{\rm u}}
\newcommand{\Ord}{{\rm Ord}}
\renewcommand{\int}{{\rm int}}

\begin{document}

\title{The unoriented band unknotting numbers of torus knots}

\author{Keisuke Himeno
}

\address{Graduate School of Advanced Science and Engineering, Hiroshima University,
1-3-1 Kagamiyama, Higashi-hiroshima, 7398526, Japan}
\email{himeno-keisuke@hiroshima-u.ac.jp}

\begin{abstract}
The unoriented band unknotting number of a knot is the minimum number of oriented or non-oriented band surgeries that turn the knot into the unknot. Batson introduced a certain non-oriented band surgery for a torus knot. The minimum number of these operations required to turn a torus knot into the unknot is called the pinch number, and it can be easily calculated from the parameters of the torus knot. In this paper, we show that the unoriented band unknotting number and the pinch number coincide for torus knots. In the proof, we use the torsion order of the unoriented knot Floer homology.
\end{abstract}

\renewcommand{\thefootnote}{}

\maketitle


\section{Introduction}\label{intro}
For a knot $K\subset S^3$, the {\it unoriented band unknotting number\/} $\u^u_b(K)$ is the minimum number of oriented or non-oriented band surgeries that turn $K$ into the unknot. 
In \cite{HNT90}, this is called the {\it $H(2)$--unknotting number\/}. 
There are several results on the unoriented band unknotting number, for example, \cite{AK14,KM09}.

In this paper, we determine the unoriented band unknotting number of the $(p,q)$--torus knot $T_{p,q}$ where $p$ and $q$ are relatively prime integers. 
In \cite{BE19}, Bettersworth and Ernst gave an upper bound of $\u^u_b(T_{p,q})$ and conjectured that this bound actually realizes $\u^u_b(T_{p,q})$. 
We give the affirmative answer to the conjecture. Note that since the unoriented band unknotting number of the mirror image of a knot is equal to that of the original knot, we may focus on the positive torus knots. 

\begin{theorem}\label{thm_main}
Let $p,q>1$ be relatively prime integers with $p$ odd. We assume $p<q$ if $q$ is odd. Suppose that 
\[
\cfrac{q}{p}=a_0+\cfrac{\pm 2}{a_1+\cfrac{\pm 1}{a_2+\cfrac{\pm 1}{a_3+\cfrac{\pm 1}{\cdots+\cfrac{\pm 1}{a_n}}}}}
\]
such that 
\begin{itemize}
\item $a_0\ge 0$,
\item $a_1$ is odd and $a_1\ge 3$, and 
\item $a_2,\ldots,a_n$ are positive even integers.
\end{itemize}
Then, $\u^u_b(T_{p,q})=n$.
\end{theorem}

A non-oriented band surgery on a torus knot considered in \cite{BE19} is also independently studied under the name of a {\it pinch move\/}, which is introduced by Batson \cite{Bat14}, and examined in detail by Jabuka and Van Cott \cite{JV21}. 
They have studied a pinch move with the motivation to investigate the non-orientable $4$-genus of a torus knot.
This operation turns a positive torus knot into another positive torus knot (possibly the unknot) and decreases the values of its parameters. 
Hence, a repeated application of a pinch move will eventually lead to the unknot. 
The minimum number of pinch moves required to transform the $(p,q)$--torus knot into the unknot is called the {\it pinch number\/}, denoted by $P(p,q)$. 
Since a pinch move is a typical non-oriented band surgery, it immediately follows $\u^u_b(T_{p,q})\le P(p,q)$. 
In fact, Theorem \ref{thm_main} claims $\u^u_b(T_{p,q})=P(p,q)$, as conjectured in \cite{BE19}.

The key ingredient to prove Theorem \ref{thm_main} is the {\it unoriented knot Floer torsion order\/} $\Ord'(K)\in\Z$ of a knot $K$ developed by Gong and Marengon in \cite{GM23}. 
According to Corollary 1.9 of \cite{GM23}, $\Ord'(K)\le \u_b^u(K)$ for any knot $K$. 
So, we have $\Ord'(T_{p,q})\le \u^u_b(T_{p,q})\le P(p,q)$. 
Thus, Theorem \ref{thm_main} is achieved by showing $\Ord'(T_{p,q})=P(p,q)$. 

The proof is based on a combinatorial discussion regarding the change induced by a pinch move on a $(1,1)$--diagram of a torus knot. 
Since a positive torus knot is an $L$--space knot (see \cite{OS05}), its $(1,1)$--diagram has a special property by the work of Greene, Lewallen and Vafaee \cite{GLV18}. 
The argument relies on the property and the fact that a torus knot can be put on the standard torus in $S^3$. 

\begin{remark}
Since a positive torus knot is an $L$--space knot, a knot Floer theoretical object is determined by its Alexander polynomial \cite{OS05}. 
Hence, we can obtain $\Ord'$ of a torus knot from its Alexander polynomial in principle. 
However, to do so, we need to get a polynomial form rather than a fractional one of its Alexander polynomial, and capture the change of the indices under a pinch move. 
It seems to be difficult, so we decided to use the $(1,1)$--diagrams.
\end{remark}

 \begin{acknowledgment}
The author would like to thank Masakazu Teragaito for his thoughtful guidance and helpful discussions about this work. The author was supported by JST SPRING, Grant Number JPMJSP2132.
\end{acknowledgment}

\section{The pinch move}
In this section, we give some facts of a pinch move and a formula for the pinch number.

Put the torus knot $T_{p,q}$ on the standard torus $\Sigma\subset S^3$ so that $T_{p,q}$ intersects all standard meridian curves of $\Sigma$ transversely. 
This is referred to as the {\it standard position\/}. 
Then there is the ``simplest" choice of a non-oriented band surgery on $T_{p,q}$, obtained by placing the band on $\Sigma$ between adjacent strings, see Figure \ref{pinch}. 
This non-oriented band surgery is called the {\it pinch move\/} on $T_{p,q}$. 
Note that the pinch move is unique up to isotopy of the band. 
The pinch move on a torus knot yields another torus knot. 
It is known how the parameters change by the pinch move. 

\begin{figure}
\centering
\includegraphics[scale=0.25]{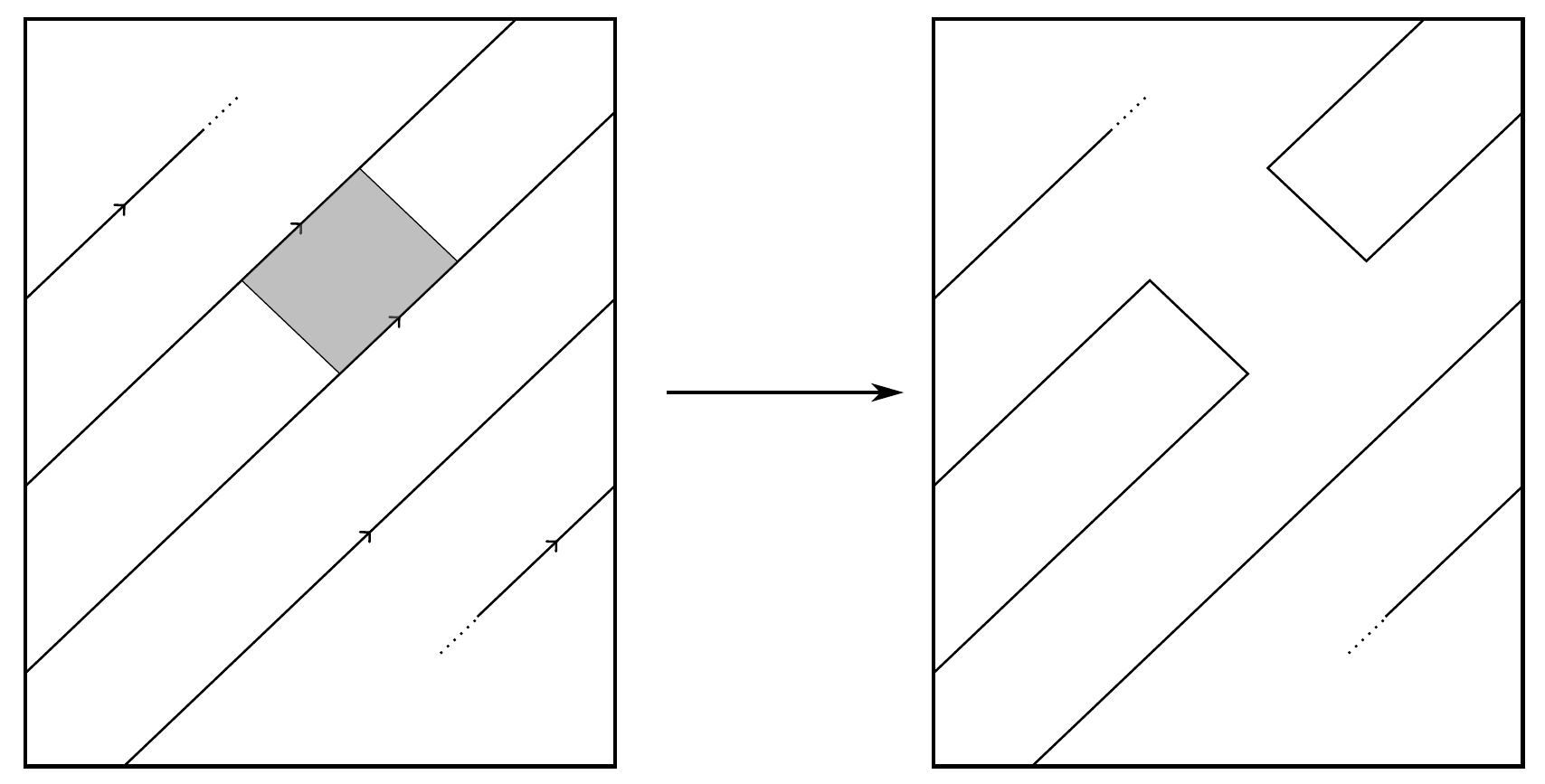}
\caption{The pinch move. Here the opposite edges of the square are identified. }\label{pinch}
\end{figure}

\begin{lemma}[\cite{Bat14,JV21}]
Let $0<p<q$ be relatively prime integers. 
The pinch move on $T_{p,q}$ yields the torus knot $T_{r,s}$ with
\[
(r,s)=(|p-2t|,|q-2h|)
\] 
where $t$ and $h$ are the integers uniquely determined by requirements
\begin{align*}
t\equiv -q^{-1} \pmod p,\ \ t\in\{0,\ldots,p-1\},\\
h\equiv p^{-1} \pmod q,\ \ h\in\{0,\ldots,q-1\}.
\end{align*}
\end{lemma}

In particular, the pinch move on a positive torus knot yields another positive torus knot. 
Then, we can get the sequence 
\[
T_{p,q}=T_{p_n,q_n}\longrightarrow T_{p_{n-1},q_{n-1}}\longrightarrow \cdots \longrightarrow T_{p_1,q_1}\longrightarrow T_{p_0,q_0}=O
\] 
for some $n\ge 1$, where $O$ is the unknot and $T_{p_{i+1},q_{i+1}}$ is obtained by the pinch move on $T_{p_i,q_i}$. 
Thus, the pinch number $P(p,q)$ is $n$ in this case. 

Here, we give a formula of $P(p,q)$ in term of a certain continued fraction expansion of $p/q$. 
We represent the pinch move from $T_{p,q}$ to $T_{r,s}$ by 
\[
T_{p,q} \overset{\epsilon}{\longrightarrow} T_{r,s},
\]
where $\epsilon={\rm Sign}(p-2t)$. 
The formula (Proposition \ref{prop_pinch_formula}) essentially follows from Jabuka and Van Cott's result:

\begin{theorem}[Theorem 2.7 of \cite{JV21}]\label{thm_JV21_pinchnumber}
Let $p,q>1$ be relatively prime integers with $p$ odd. We assume $p<q$ if $q$ is odd.
Suppose that there is a sequence
\[
T_{p,q}=T_{p_n,q_n}\overset{\epsilon_n}{\longrightarrow} T_{p_{n-1},q_{n-1}}\overset{\epsilon_{n-1}}{\longrightarrow} \cdots \overset{\epsilon_2}{\longrightarrow}T_{p_1,q_1}\overset{\epsilon_1}{\longrightarrow} T_{p_0,q_0}=O.
\]
Then the integers generated in the process of these pinch moves satisfy the following five identities and constraints:
\begin{enumerate}
\item $p_0=1$. 
\item $q_0\ge 0$, and if $q_0=0$ or $1$, then $\epsilon_1=-1$.
\item The value of $p_1$ is odd and $p_1\ge 3$.
\item $q_1=q_0p_1-2\epsilon_1$.
\item There exist positive even integers $m_1,m_2,\ldots,m_{n-1}$ such that, for all $k\in \{2,3,\ldots,n\}$,
\[
p_k=m_{k-1}p_{k-1}-\epsilon_{k-1}\epsilon_{k}p_{k-2},\ q_k=m_{k-1}q_{k-1}-\epsilon_{k-1}\epsilon_{k}q_{k-2}.
\]
\end{enumerate}
Moreover, the correspondence between $(p,q)$ and the data set
\[
\{n, q_0,p_1,\{\epsilon_k\}_{k=1}^n,\{m_k\}_{k=1}^{n-1}\}
\]
subject to constraints $(1)$--$(5)$ above is a bijection.
\end{theorem}

Note that the roles of $p$ and $q$ in the above theorem are swapped from the original ones.

\begin{proposition}\label{prop_pinch_formula}
Let $p,q>1$ be relatively prime integers with $p$ odd. We assume $p<q$ if $q$ is odd. 
Suppose that 
\[
\cfrac{q}{p}=q_0+\cfrac{\pm 2}{p_1+\cfrac{\pm 1}{m_1+\cfrac{\pm 1}{m_2+\cfrac{\pm 1}{\cdots+\cfrac{\pm 1}{m_{n-1}}}}}}
\]
such that 
\begin{itemize}
\item $q_0\ge 0$,
\item $p_1$ is odd and $p_1\ge 3$, and 
\item $m_1,m_2,\ldots,m_{n-1}$ are positive even integers.
\end{itemize}
Then, $P(p,q)=n$. (If $\frac{q}{p}=q_0+\frac{\pm 2}{p_1}$, then $P(p,q)=1$.)
\end{proposition}
\begin{proof}
By Theorem \ref{thm_JV21_pinchnumber}, we can obtain
\[
\cfrac{q_n}{p_n}=q_0+\cfrac{-2\epsilon_1}{p_1+\cfrac{-\epsilon_1\epsilon_2}{m_1+\cfrac{-\epsilon_2\epsilon_3}{m_2+\cfrac{-\epsilon_3\epsilon_4}{\cdots+\cfrac{-\epsilon_{n-1}\epsilon_n}{m_{n-1}}}}}}.
\]
Hence, there exists the continued fraction expansion as required. 

Conversely, given the continued fraction satisfying the conditions above, the data set $\{n, q_0,p_1,\{\epsilon_k\}_{k=1}^n,\{m_k\}_{k=1}^{n-1}\}$ is determined, since $\epsilon_k=\pm 1\ (k=1,\ldots,n)$. 
Note that the correspondence between $(p,q)$ and the data set is a bijection by Theorem \ref{thm_JV21_pinchnumber}.
Therefore, the continued fraction satisfying the conditions is unique, so we have the conclusion.  
\end{proof}

\section{The unoriented knot Floer homology}
In this section, we review the unoriented knot Floer homology. 
We assume that the reader is familiar with knot Floer homology theory. 
There are good survey articles, for example \cite{Hom17,Man16}. 

The unoriented knot Floer homology is introduced in \cite{OSS17A} (see also \cite{Fan19,OSS17B}). 
Let $(\Sigma_g; \boldsymbol\alpha=\{\alpha_1,\ldots,\alpha_g\},\boldsymbol\beta=\{\beta_1,\ldots,\beta_g\};z,w)$ be a genus $g$ doubly pointed Heegaard diagram for a knot $K$ in $S^3$. 
We take the symmetric product space ${\rm Sym}^g(\Sigma_g)$ and its subspaces $T_{\boldsymbol\alpha}:=\alpha_1\times\cdots\times\alpha_g$ and $T_{\boldsymbol\beta}:=\beta_1\times\cdots\times\beta_g$. 
The {\it unoriented knot Floer chain complex\/} $\CFK'(K)$ is an $\F_2[U]$--module freely generated by intersection points $T_{\boldsymbol\alpha}\cap T_{\boldsymbol\beta}$, where $\F_2[U]$ is a polynomial ring with the formal variable $U$ and $\F_2:=\Z/2\Z$ coefficients. 
Let $\pi_2(x,y)$ be the set of homotopy classes of Whitney disks from $x$ to $y$. 
The differential $\partial$ is defined by 
\[
\partial x=\sum_{y\in T_{\boldsymbol\alpha}\cap T_{\boldsymbol\beta}}\sum_{\substack{\phi\in \pi_2(x,y),\\ \mu(\phi)=1}}\#(\mathcal{M}(\phi)/\R)U^{n_z(\phi)+n_w(\phi)}\cdot y.
\] 
Here, $\mathcal{M}(\phi)$ is the space of holomorphic representations of $\phi$, $\mu(\phi)$ is the Maslov index and $n_z(\phi)$ (resp. $n_w(\phi)$) is the algebraic intersection number between $\phi$ and $\{z\}\times{\rm Sym}^{g-1}(\Sigma_g)$ (resp. $\{w\}\times{\rm Sym}^{g-1}(\Sigma_g)$). 
Remark that $\partial$ differs from the differential of $\CFK^{\infty}(K)$ only in the way that the powers of $U$ are counted. 
The grading structure referred to as {\it $\delta$-grading\/} can be considered, but it is not relevant to this discussion, so we do not need to worry about it.

The {\it unoriented knot Floer homology\/} $\HFK'(K)$ is defined as the homology of $(\CFK'(K),\partial)$. 
Of course, $\HFK'(K)$ has an $\F_2[U]$--module structure.
In \cite{GM23}, they define the {\it unoriented knot Floer torsion order\/} as 
\[
\Ord'(K)=\min\{n\ge 0\mid U^n\cdot {\rm Tor}=\{0\}\}
\]
where ${\rm Tor}$ is the $\F_2[U]$--torsion submodule of $\HFK'(K)$.
As mentioned in Section \ref{intro}, the following result is important to prove Theorem \ref{thm_main}.

\begin{proposition}[Corollary 1.9 of \cite{GM23}]\label{prop_GM}
For a knot $K$ in $S^3$,
\[
\Ord'(K)\le \u^u_b(K).
\]
\end{proposition}

\begin{remark}
For the torsion order, the original symbol used in \cite{GM23} is $\Ord_U(K)$.
\end{remark}

In general, it is difficult to calculate the differential $\partial$ of $\CFK'(K)$. 
However, if $g=1$, that is, the underlying surface $\Sigma_g$ of a doubly pointed Heegaard diagram is a torus, we can combinatorially calculate $\partial$. 
(Such a diagram is also called a {\it $(1,1)$--diagram\/}.) 
Specifically, for $\phi\in \pi_2(x,y)$, the condition $\mu(\phi)=1$ implies $\#(\mathcal{M}(\phi)/\R)=1\ \pmod2$. 
Also, disks $\phi\in\pi_2(x,y)$ satisfying $\mu(\phi)=1$ can be easily found by considering {\it bigons\/} cobounded by lifts of $\alpha$ and $\beta$ under the universal covering map $\R^2\to\Sigma_1$. 
For more details, see \cite{GMM05,GLV18,OS04}.

\section{A $(1,1)$--diagram}
In this section, we give a $(1,1)$--diagram of a torus knot, where the torus knot itself is also placed on the underlying torus. 
Then, we examine change of the $(1,1)$--diagram by the pinch move. 
Hereafter, we assume that $1<p<q$. 

\subsection{A $(1,1)$--diagram of a torus knot}
Let $K$ be a knot in $S^3$, $\Sigma$ be the torus, $\alpha,\beta$ be simple closed curves in $\Sigma$ and $z,w$ be two points in $\Sigma$ disjoint from $\alpha,\beta$. 
A tuple $(\Sigma;\alpha,\beta;z,w)$ is called a {\it $(1,1)$--diagram\/} (or, a {\it doubly pointed Heegaard diagram\/}) of $K$ if
\begin{itemize}
\item $(\Sigma;\alpha,\beta)$ is the standard genus one Heegaard diagram of $S^3$,
\item $K=k_\alpha\cup k_\beta$, where $k_\alpha$ (resp. $k_\beta$) is a trivial arc connecting $z,w\in\Sigma$ in the inner (resp. outer) solid torus such that it is disjoint from the meridian disk $D_{\alpha}$ (resp. $D_{\beta}$) bounded by $\alpha$ (resp. $\beta$).
\end{itemize}

It is well known that any torus knot admits a $(1,1)$--diagram, and it is given as follow:
First, we draw the $(p,q)$--torus knot $T_{p,q}$ in the standard position on a torus $\Sigma$. 
Let $\alpha$ and $\beta$ be a standard meridian and a standard longitude of $\Sigma$, respectively. 
Furthermore, we orient the knot $T_{p,q}$.
Second, take a point $o$ on the knot $T_{p,q}$ near $\alpha$, but not $\alpha\cup\beta$. 
Then, move $o$ on $T_{p,q}$ along the orientation of $T_{p,q}$ just before touching $\alpha$. 
The trajectory of $o$ is denoted as $k_\alpha$, and let $z$ and $w$ be endpoints of $k_\alpha$, the initial point is $z$ and the terminal point is $w$. See Figure \ref{diagram_T35_1_2}. 
Finally, move $o$ on $T_{p,q}$ further from $w$ along the orientation of $T_{p,q}$ until it reaches $z$. 
However, at this time, $\beta$ is also moved so that $o$ does not touch $\beta$. 
Then, the trajectory of $o$ is denoted as $k_\beta$. See Figure \ref{diagram_T35_3_4}.

\begin{figure}
\centering
\includegraphics[scale=0.3]{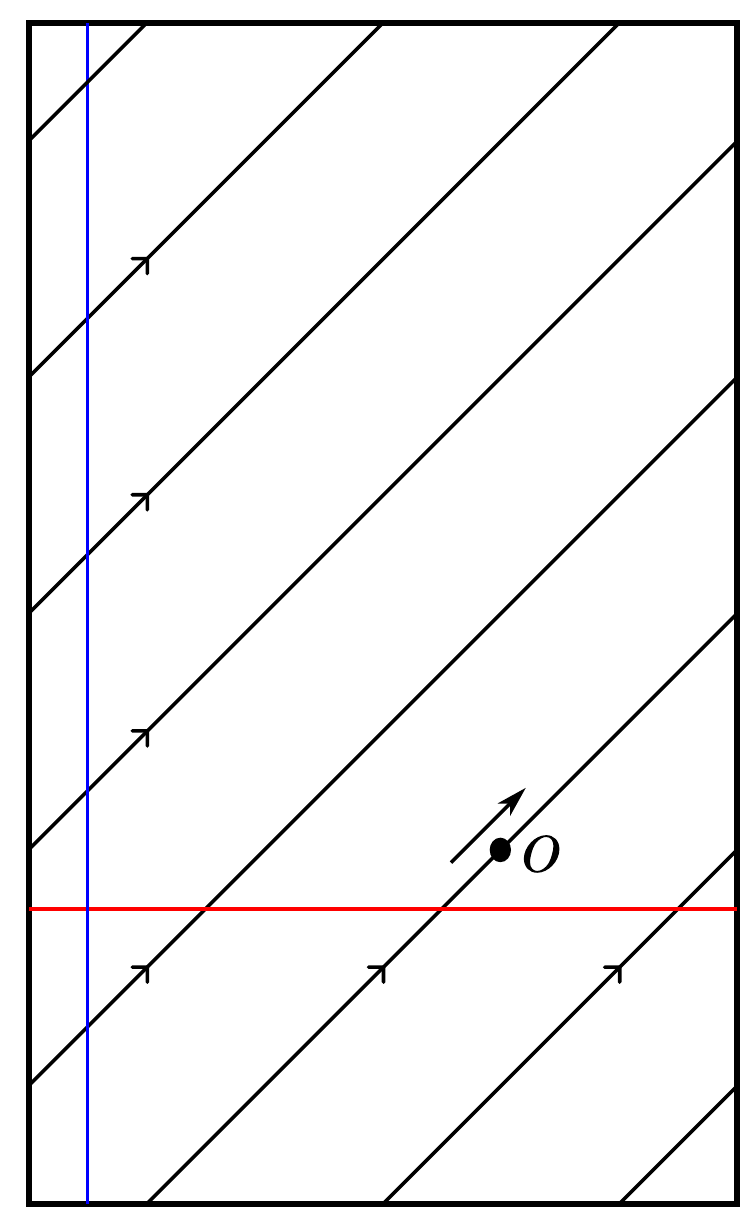}
\hspace{10mm}
\includegraphics[scale=0.3]{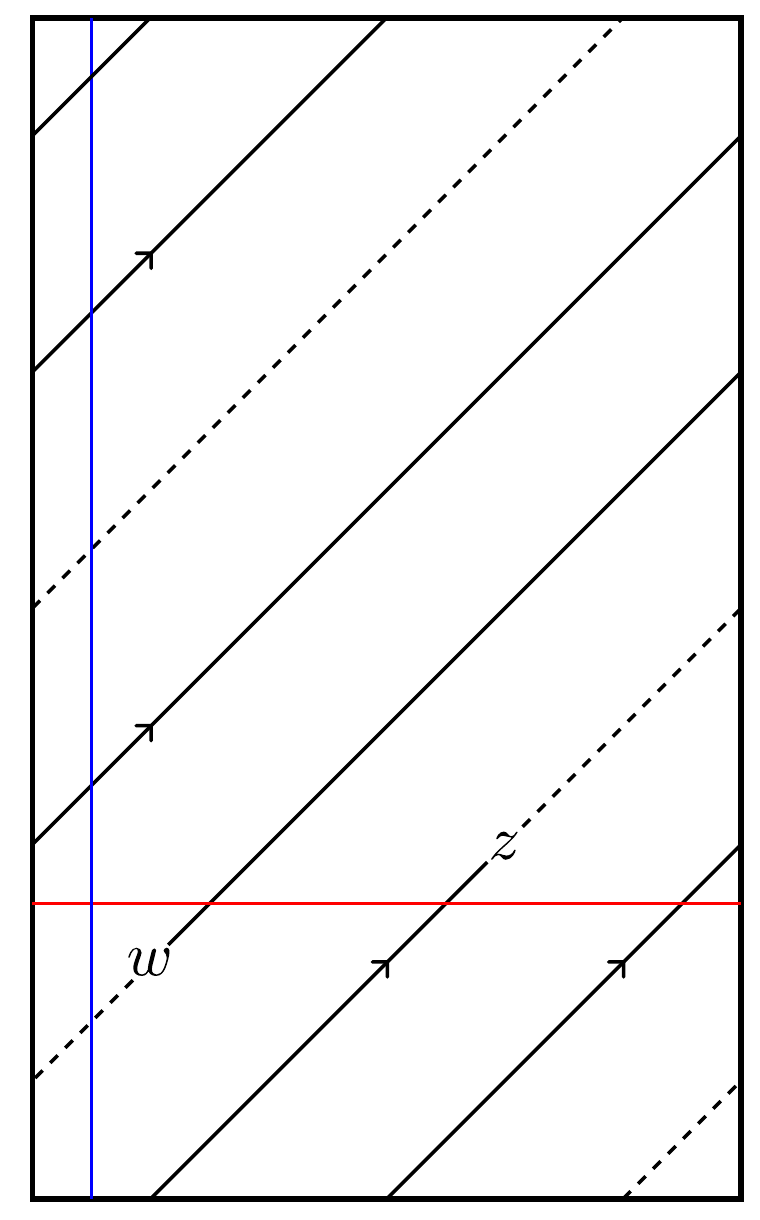}
\caption{(Left) A fundamental domain of the torus $\Sigma$ with the torus knot $T_{p,q}$ (in this case $T_{3,5}$) depicted, where the opposite edges are identified.
The red and blue curves are $\alpha$ and $\beta$ respectively. 
The point $o$ at the initial position is represented by the black dot. 
Move it in the direction of the arrow. 
(Right) Stop the point $o$ just before touching $\alpha$. 
The arc $k_\alpha$ is represented by the dashed line.}\label{diagram_T35_1_2}
\end{figure}

\begin{figure}
\centering
\includegraphics[scale=0.3]{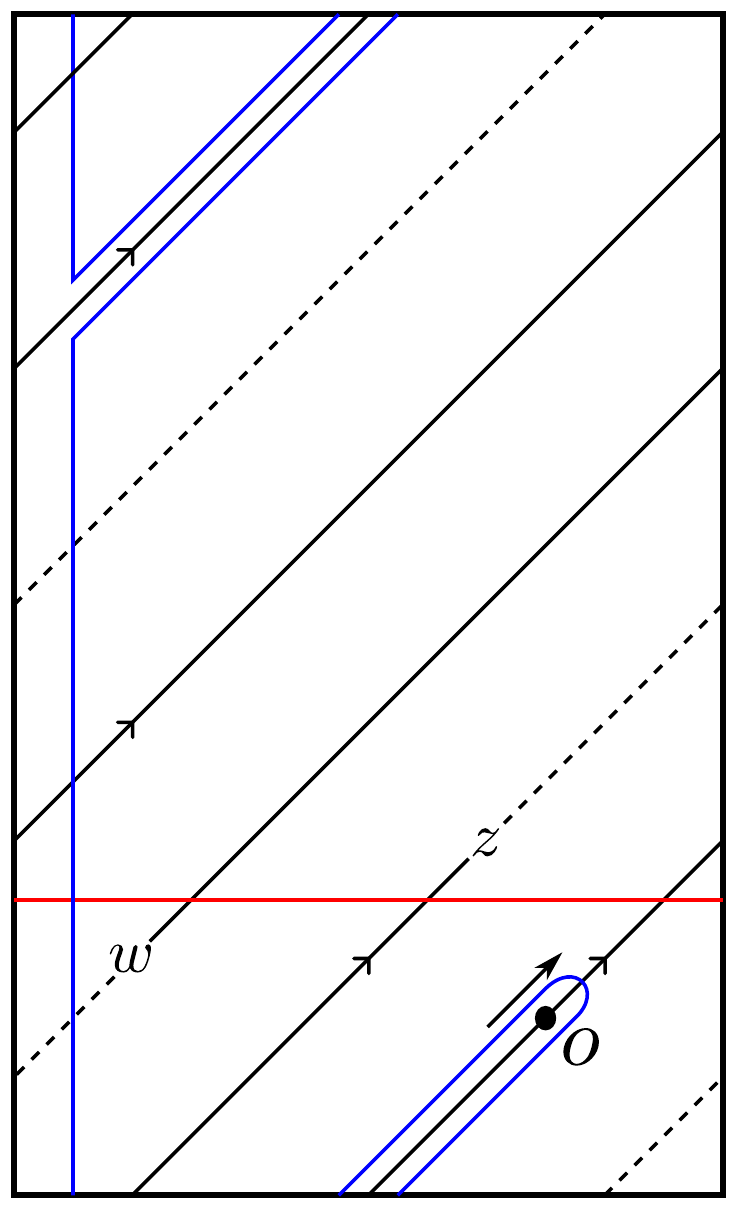}
\hspace{10mm}
\includegraphics[scale=0.3]{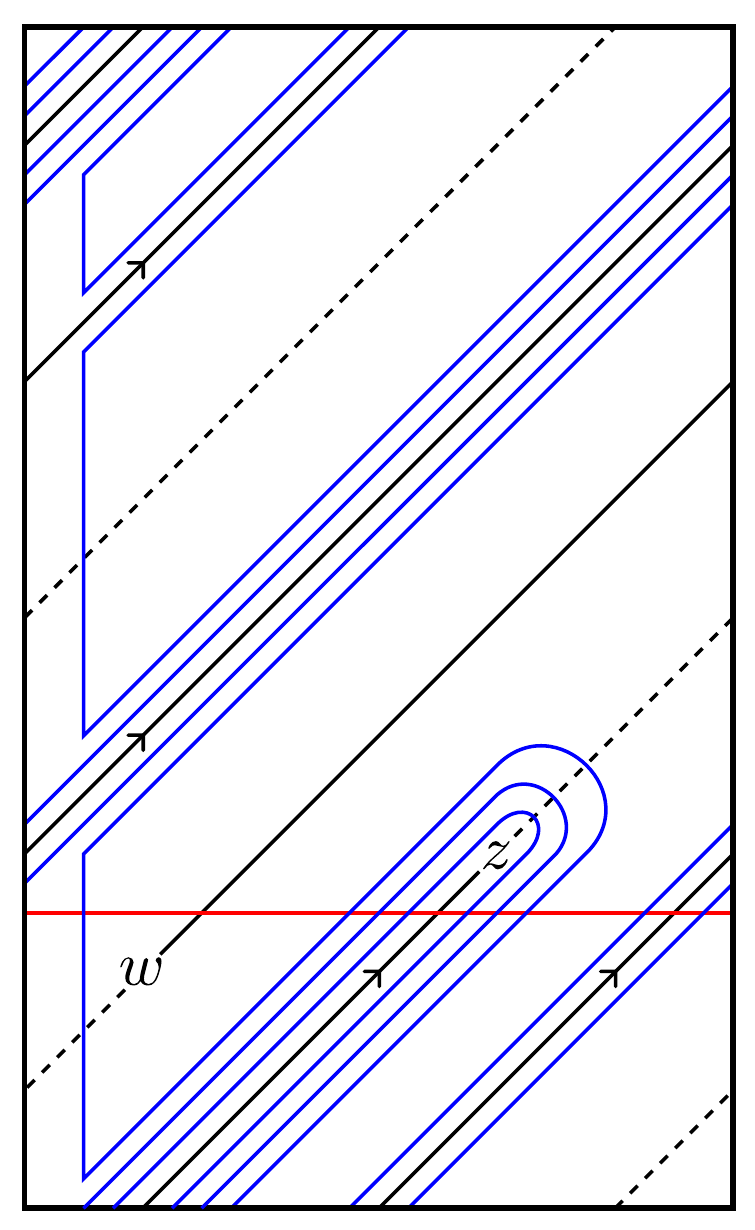}
\caption{(Left) Moving $o$ from $w$ toward $z$. (Right) The point $o$ reaches $z$. 
$k_\beta$ is represented by the solid line.}\label{diagram_T35_3_4}
\end{figure}

Thus, $(\Sigma;\alpha,\beta;z,w)$ is a $(1,1)$--diagram of the torus knot $T_{p,q}=k_\alpha\cup k_\beta$ after pushing $k_\alpha$ and $k_\beta$ slightly off from $\Sigma$. 
However, we place these two arcs $k_\alpha$ and $k_\beta$ on $\Sigma$ in our argument, 
and call it the $(1,1)$--diagram of $T_{p,q}$. 
It should be emphasized that 
$k_\alpha$ (resp. $k_\beta$) is disjoint from the curve $\alpha$ (resp. $\beta$). 
In general, this diagram is non-reduced, that is, there exists a bigon cobounded by $\alpha$ and $\beta$ that does not contain $z$ or $w$. 
So, move $\beta$ further to eliminate such bigons. 
See Figure \ref{diagram_T35_5_6}. 
As a result, we get a reduced $(1,1)$--diagram. 
Note again that $T_{p,q}=k_\alpha\cup k_\beta$ is located on the torus $\Sigma$ and $k_\alpha\cap k_\beta=\{z,w\}$. 

\begin{figure}
\centering
\includegraphics[scale=0.3]{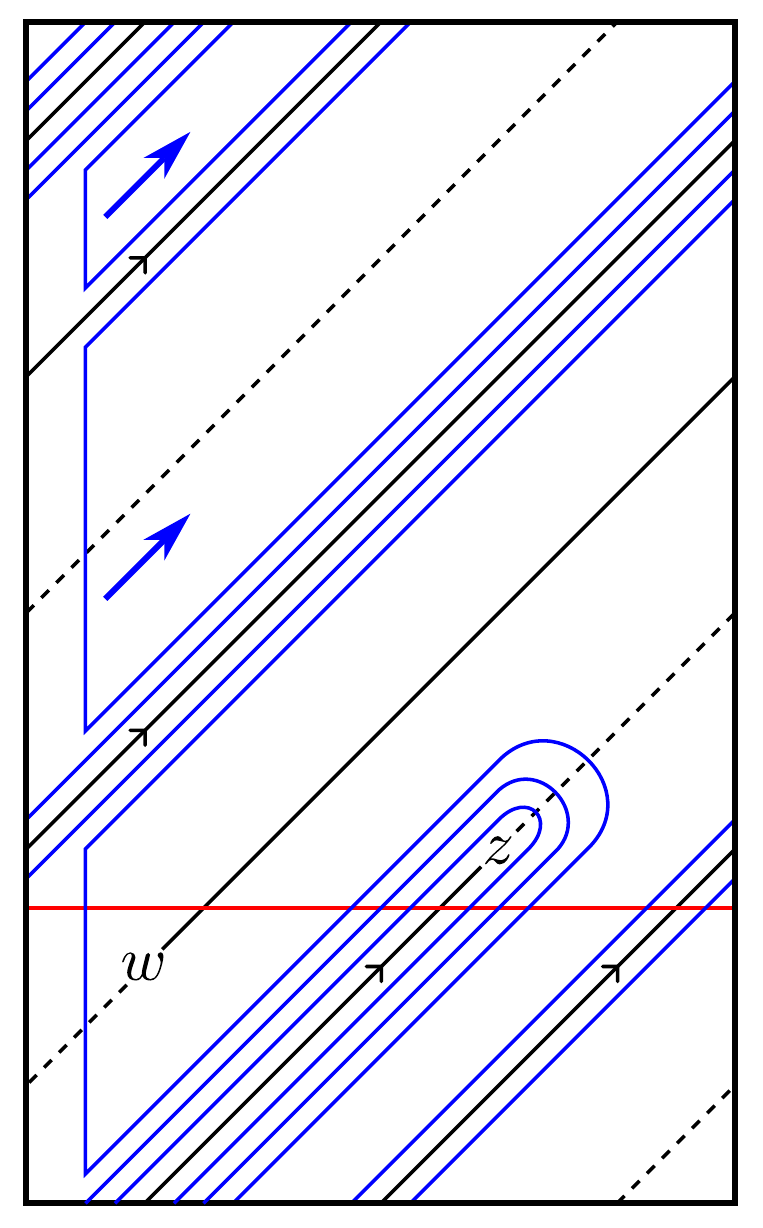}
\hspace{10mm}
\includegraphics[scale=0.3]{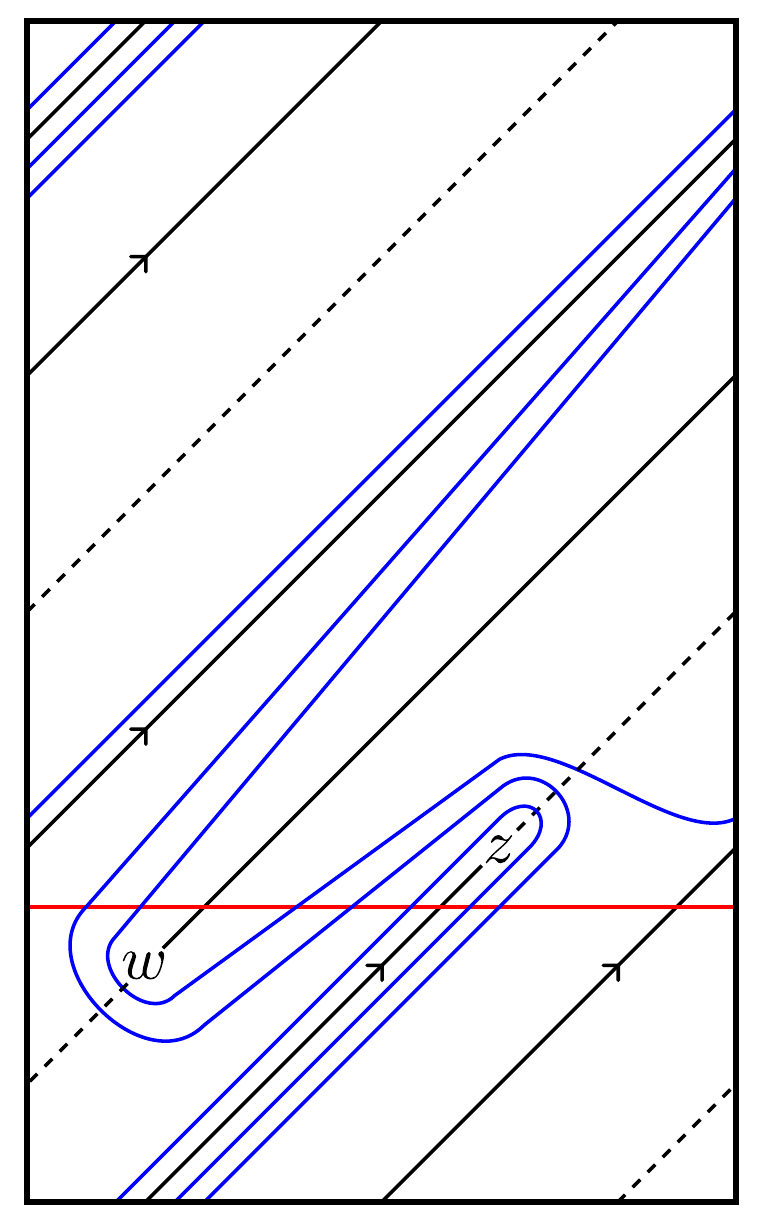}
\caption{The curve $\beta$ in the left diagram moves according to the blue arrows, resulting in the right diagram that is a reduced $(1,1)$--diagram.}\label{diagram_T35_5_6}
\end{figure}

When $\beta$ is cut by $\alpha$, there are two types of arcs; an arc that goes around in the longitudinal direction on $\Sigma$ is called a {\it straight arc\/}, while other arcs are called {\it rainbow arcs\/} (see Figure \ref{arc_orient}). 
Both type of arcs are referred to as $\beta$--arcs. 
In \cite{GLV18}, they state that a reduced $(1,1)$--diagram $(\Sigma;\alpha,\beta;z,w)$ presents an $L$--space knot if and only if there exist orientations on $\alpha$ and $\beta$ that induce coherent orientations on the boundary of every bigon $(D,\partial D)\subset (\Sigma,\alpha\cup\beta)$. 
Therefore, we can orient $\alpha$ and the rainbow arcs of the $(1,1)$--diagram of $T_{p,q}$ as shown in Figure \ref{arc_orient}, since a positive torus knot is an $L$--space knot. 
From now on, we will consider only the $(1,1)$--diagram $(\Sigma;\alpha,\beta;z,w)$ of $T_{p,q}=k_\alpha\cup k_\beta$ obtained in this way. 

\begin{figure}
\centering
\includegraphics[scale=0.4]{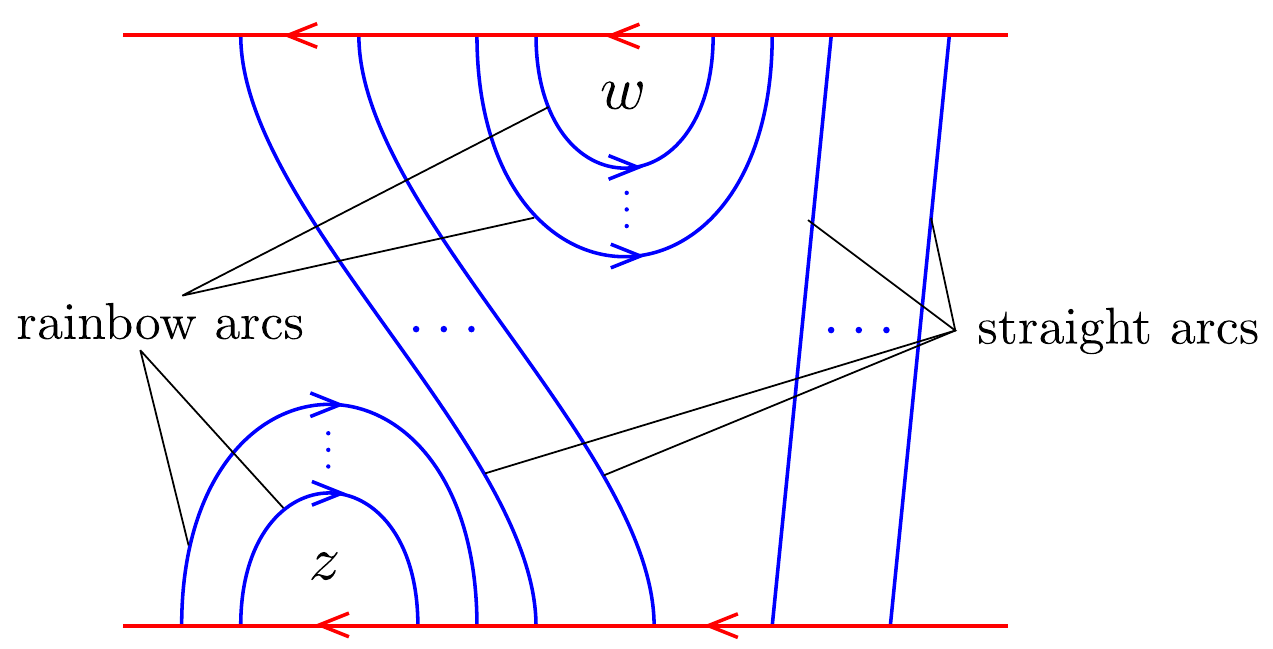}
\caption{Orient $\alpha$ (with its copy) and the rainbow arcs.}\label{arc_orient}
\end{figure}

In the following figures, the red (resp. blue, dashed, solid) line represents $\alpha$ (resp. $\beta$, $k_\alpha$, $k_\beta$). 

\subsection{Regions}\label{Regions}

The position of $T_{p,q}=k_\alpha\cup k_\beta$ in the $(1,1)$--diagram will be important later. 
By cutting the underlying surface $\Sigma$ along $\alpha$ and $\beta$, we obtain the connected components, called {\it regions\/}, like those in Figure \ref{region} (not all of them will necessarily appear). 
Each region is cobounded by $\beta$--arcs and some arcs derived from $\alpha$, where the latter ones are called $\alpha$--arcs. 
We assign names to some regions; 
The region of $(6)$ and $(7)$ are the {\it $A$--\/} and {\it $Y$--regions\/}, respectively. 
Also, the region of $(8)$ is the {\it $H$--region\/}. 
When the $H$--region exists, there is neither $A$-- nor $Y$--region, and when the $H$--region is absent, there is exactly one each in $A$-- and $Y$--region.

\begin{figure}
\centering
\includegraphics[scale=0.35]{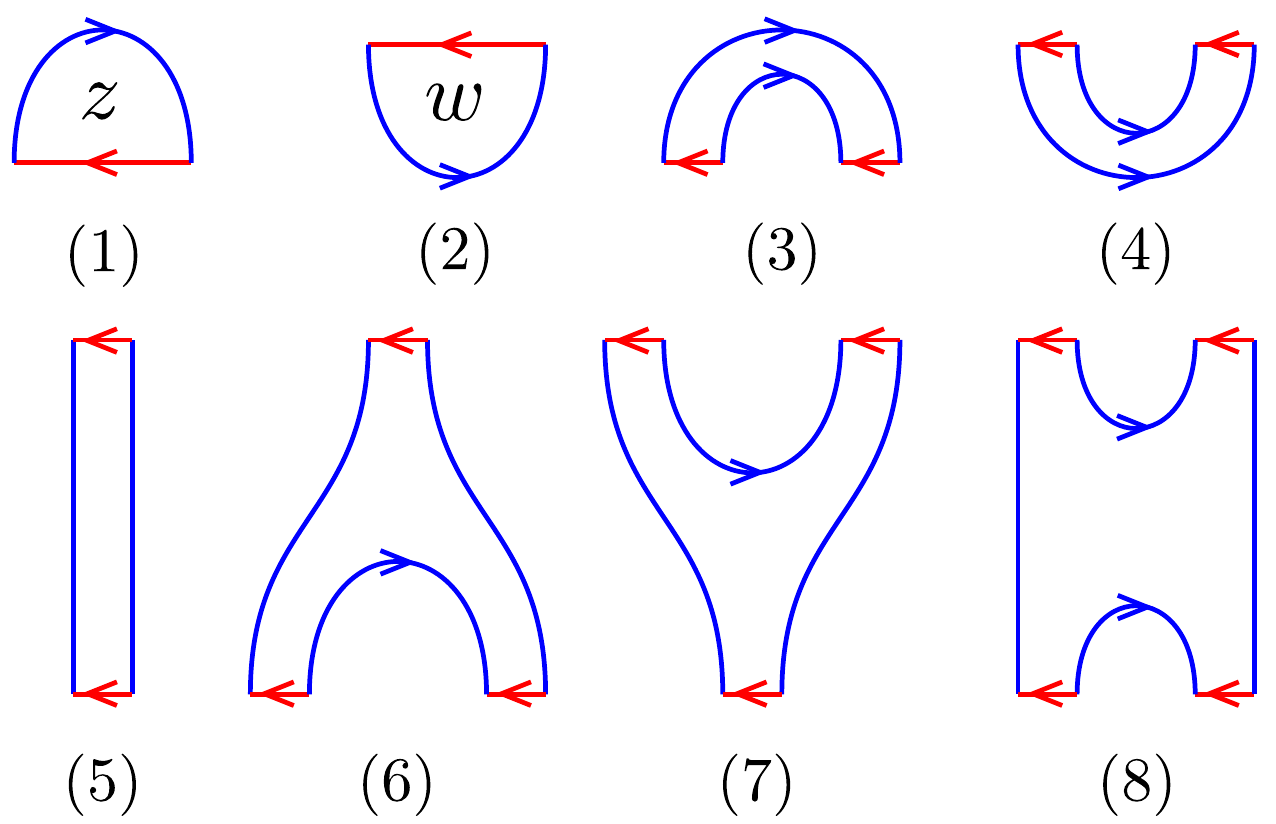}
\caption{Connected components that appear when $\Sigma$ is cut by $\alpha$ and $\beta$. 
If the orientation of an arc is determined, it is drawn by an arrow.}\label{region}
\end{figure}

\begin{figure}
\centering
\includegraphics[scale=0.35]{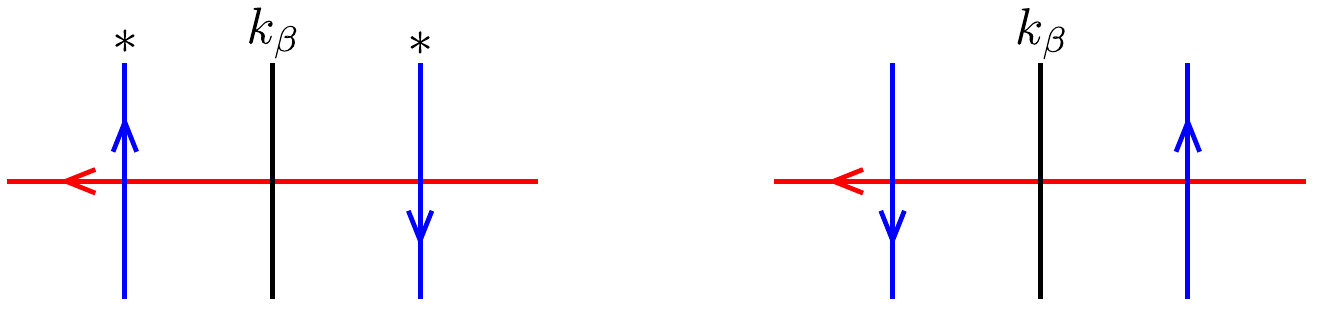}
\caption{Two adjacent $\beta$--arcs on an $\alpha$--arc with opposite orientation. }\label{fig_lem_opposite_edge}
\end{figure}

\begin{lemma}\label{lem_opposite_edge}
If two adjacent $\beta$--arcs on an $\alpha$--arc have opposite orientations, then $k_\beta$ passes through the space between them at least once (see Figure \ref{fig_lem_opposite_edge}).
\end{lemma}
\begin{proof}
We prove only the situation on the left of Figure \ref{fig_lem_opposite_edge}, and the remaining case can be shown in the same manner. 

If the two $\beta$--arcs are connected by a rainbow arc over $z$ in the direction of $*$, we are done. 
If not, only a $Y$--region can connect those in the direction of $*$, because of the orientation of the rainbow arcs. Then, the orientation of $\beta$--arcs at the two $\alpha$--arcs on the top of the $Y$-region is the same situation as the left of Figure \ref{fig_lem_opposite_edge}. 
Hence, these $\beta$--arcs on the top of the $Y$--region are connected by rainbow arcs over $z$, so we have the conclusion.
\end{proof}

Note that both of the $A$-- and $Y$--regions, if they exist, meet $k_\alpha$, because $k_\alpha$ connects the two points $z$ and $w$, and cannot meet $\alpha$. Similarly for the $H$--region.

\begin{figure}
\centering
\includegraphics[scale=0.35]{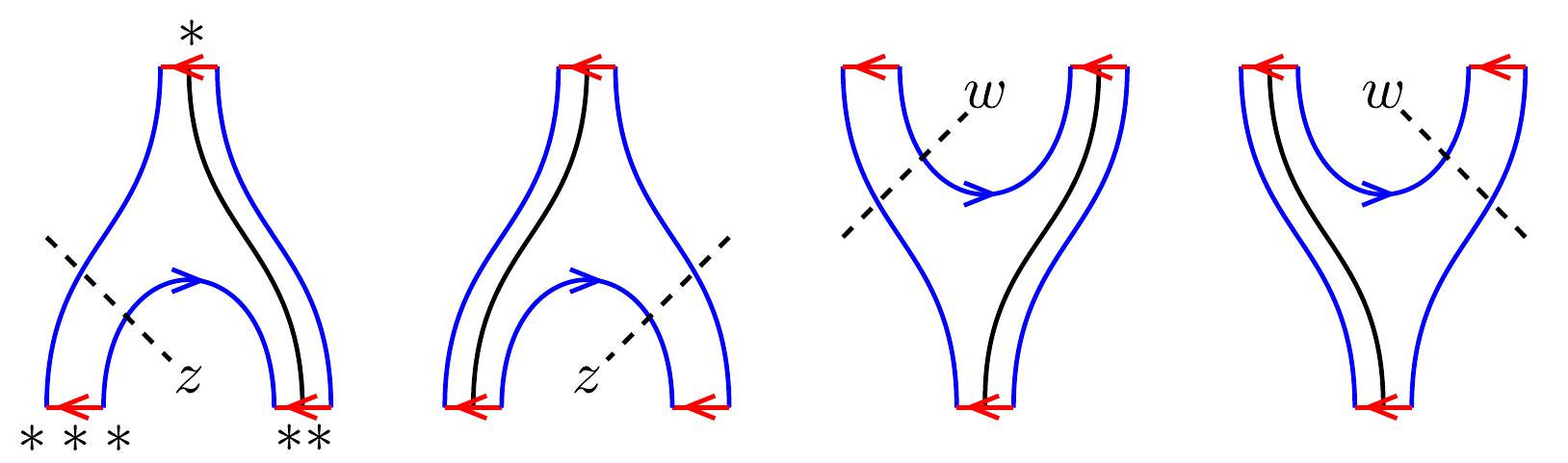}
\caption{The $A$-- and $Y$--regions with $k_\alpha$. $k_\alpha$ and $k_\beta$ are represented by the dashed line and the solid line, respectively.}\label{fig_lem_Y}
\end{figure}

\begin{lemma}\label{lem_Y}
For the $A$-- or $Y$--region with $k_\alpha$ as shown in Figure \ref{fig_lem_Y}, $k_\beta$ passes through it as illustrated there.
\end{lemma}
\begin{proof}
We prove only the left-most case of Figure \ref{fig_lem_Y}, and the remaining cases can be shown in a similar manner. 

First, if the right straight arc is oriented upward, we are done by applying Lemma \ref{lem_opposite_edge} to $**$. Assume that the right straight arc is oriented downward. 
If the left straight arc is oriented upward, then we are done by applying Lemma \ref{lem_opposite_edge} to $*$. Lastly, if the left straight arc is oriented downward, we can apply Lemma \ref{lem_opposite_edge} to $*\!*\!*$. However, this contradicts the fact that $k_\alpha$ and $k_\beta$ do not intersect.
\end{proof}

\begin{figure}
\centering
\includegraphics[scale=0.35]{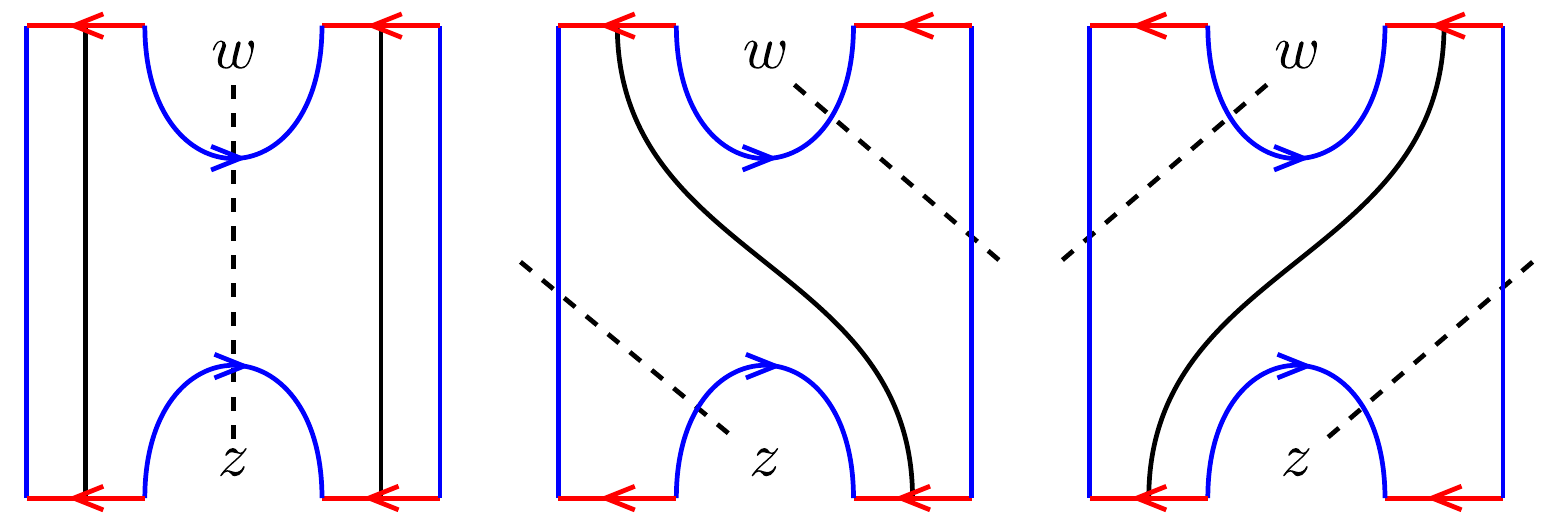}
\caption{The $H$--region with $k_\alpha$.}\label{fig_lem_H}
\end{figure}

\begin{lemma}\label{lem_H}
For the $H$--region with $k_\alpha$ as shown in Figure \ref{fig_lem_H}, $k_\beta$ passes through it as illustrated there.
\end{lemma}
\begin{proof}
Just as in the proof of Lemma \ref{lem_Y}, it can be shown using Lemma \ref{lem_opposite_edge}.
\end{proof}

\subsection{The pinch move and the $(1,1)$--diagram}
In this subsection, we give an observation of a change of the $(1,1)$--diagram $(\Sigma;\alpha,\beta;z,w)$ of $T_{p,q}=k_\alpha\cup k_\beta$ by operating the pinch move. 
In conclusion, it can be seen that $\alpha$ and $\beta$ remain stationary while only the two points $z$ and $w$ move. 

The neighborhood of $\alpha$ in the $(1,1)$--diagram looks like Figure \ref{pinch_diagram_1}. 
First, move $z$ along $k_\beta$ to the string with $*$ as in Figure \ref{pinch_diagram_1}, which is the ``right" of the line with $w$. 
(If there is no such string, then we do not move $z$.) 
In doing so, $\alpha$ should move in such a way that it does not touch $z$, and $k_\alpha$ should be extended, see Figure \ref{pinch_diagram_2}.

\begin{figure}
\centering
\includegraphics[scale=0.4]{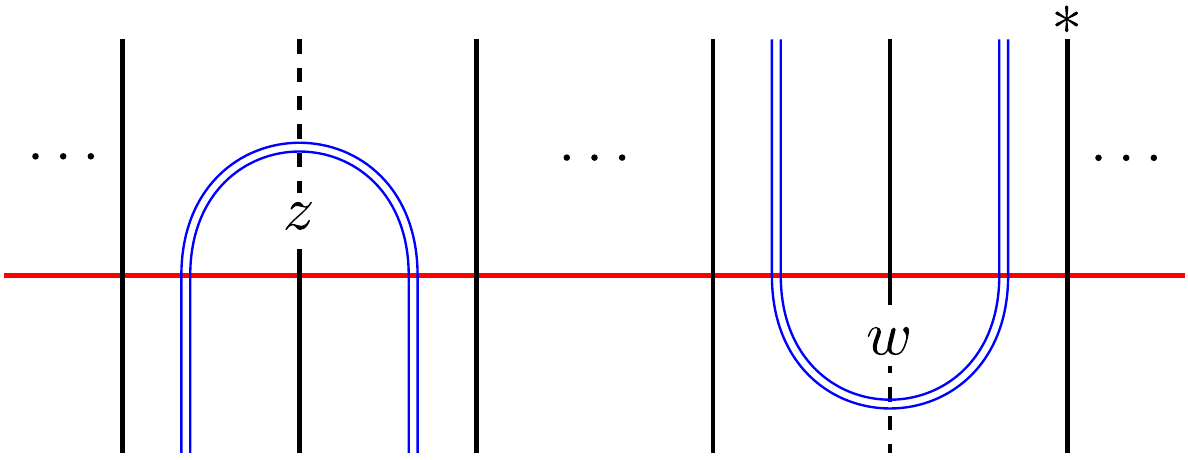}
\caption{The neighborhood of $\alpha$ in the $(1,1)$--diagram. 
Double lines represent several parallel lines (including $0$ or $1$ line). 
The straight $\beta$--arcs are actually present, but they have been omitted. }\label{pinch_diagram_1}
\end{figure}

\begin{figure}
\centering
\includegraphics[scale=0.4]{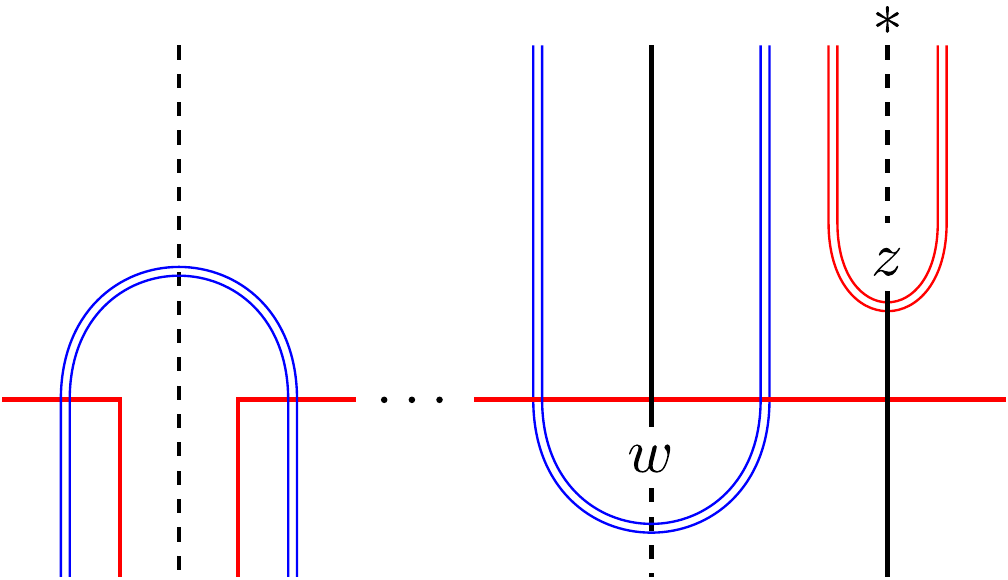}
\caption{Move $z$ to the line with $*$. During this process, $z$ may push $\alpha$ several times.}\label{pinch_diagram_2}
\end{figure}

Second, we perform the pinch move on $T_{p,q}$ along the shaded band in the left of Figure \ref{pinch_diagram_3}. Then, we newly put two points $z'$ and $w'$ on the resulting knot, and set $k'_\alpha$ and $k'_\beta$ as indicated in Figure \ref{pinch_diagram_3}. 
Also, isotope $\alpha$ to $\alpha'$ as shown there, and set $\beta'=\beta$.

\begin{figure}
\centering
\includegraphics[scale=0.35]{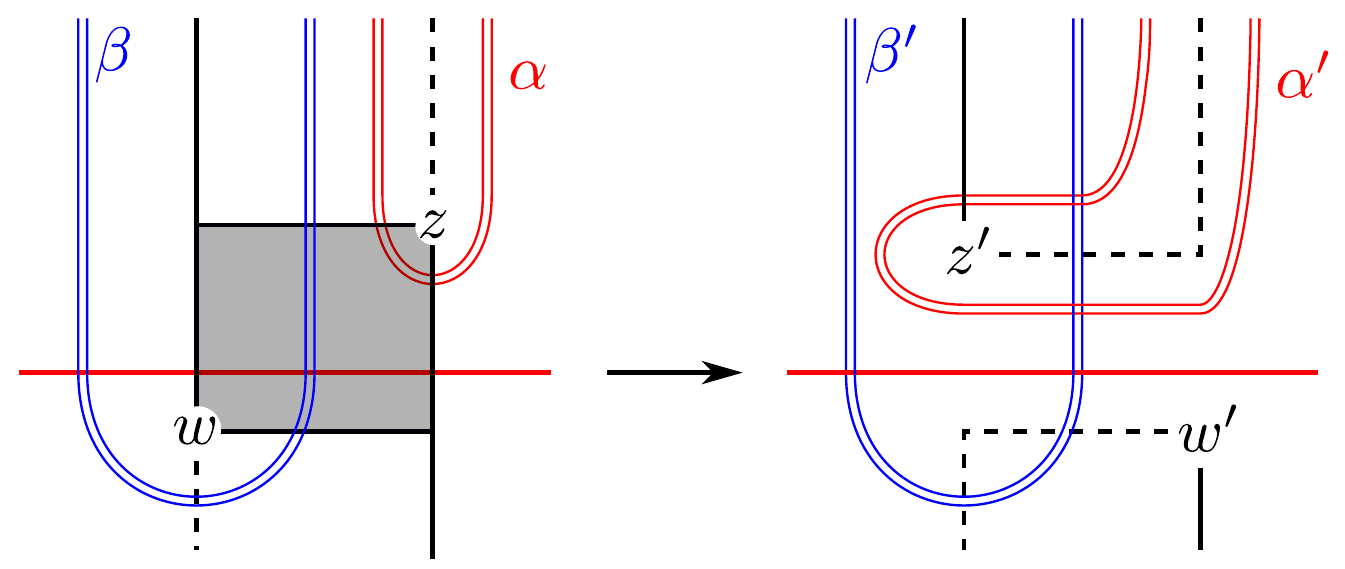}
\caption{The pinch move changes the situation from left to right. The dashed line is $k'_\alpha$ and the solid line $k'_\beta$.}\label{pinch_diagram_3}
\end{figure}

\begin{lemma}
The tuple $(\Sigma;\alpha',\beta';z',w')$ is a (non-reduced) $(1,1)$--diagram of a torus knot $T_{p',q'}=k'_\alpha\cup k'_\beta$, which is the knot obtained by the pinch move on $T_{p,q}$.
\end{lemma}
\begin{proof}
Since $\alpha'$ is isotopic to $\alpha$ and $\beta'$ is equal to $\beta$ as a set, the tuple $(\Sigma;\alpha',\beta')$ is a Heegaard diagram of $S^3$. 
Moreover, $k'_\alpha$ and $k'_\beta$, after pushing off from $\Sigma$, are trivial arcs connecting $z',w'\in\Sigma$ in each handlebody, and they are disjoint from the meridian disks bounded by $\alpha'$ and $\beta'$, respectively. 
\end{proof}

Third, move $\alpha'$ and $k'_{\alpha}$ toward $k'_\beta$ as shown in Figures \ref{pinch_diagram_4_new1} and \ref{pinch_diagram_4_new2}. We stop when the tongue-shaped part of $\alpha'$ swallows $k'_\beta$.
(If the situation on the right of Figure \ref{pinch_diagram_4_new2} happens, we continue to move $\alpha'$ toward the original $k_\alpha$, which was above $w'$, as there. After that, the right of Figure \ref{pinch_diagram_4_new1} occurs, because of how to pass $k'_\alpha$ and $k'_\beta$. (see also the bottom of Figure \ref{pinch_diagram_5}). In this case, we stop then.)

\begin{figure}
\centering
\includegraphics[scale=0.35]{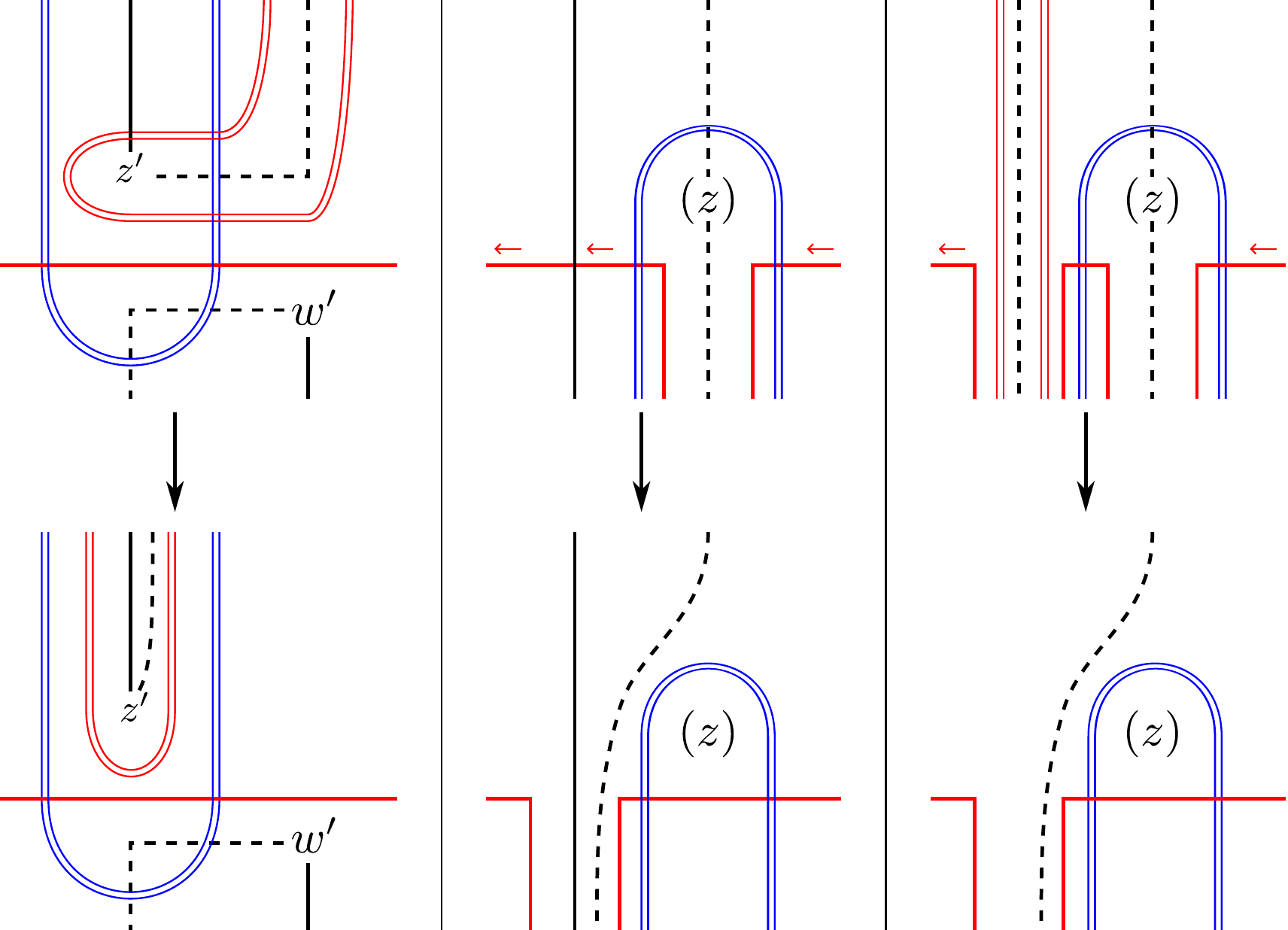}
\caption{Move $\alpha'$ and $k'_{\alpha}$. The situation of the neighborhood of the original $z$, indicated by $(z)$, is either the middle or the right.}\label{pinch_diagram_4_new1}
\end{figure}

\begin{figure}
\centering
\includegraphics[scale=0.35]{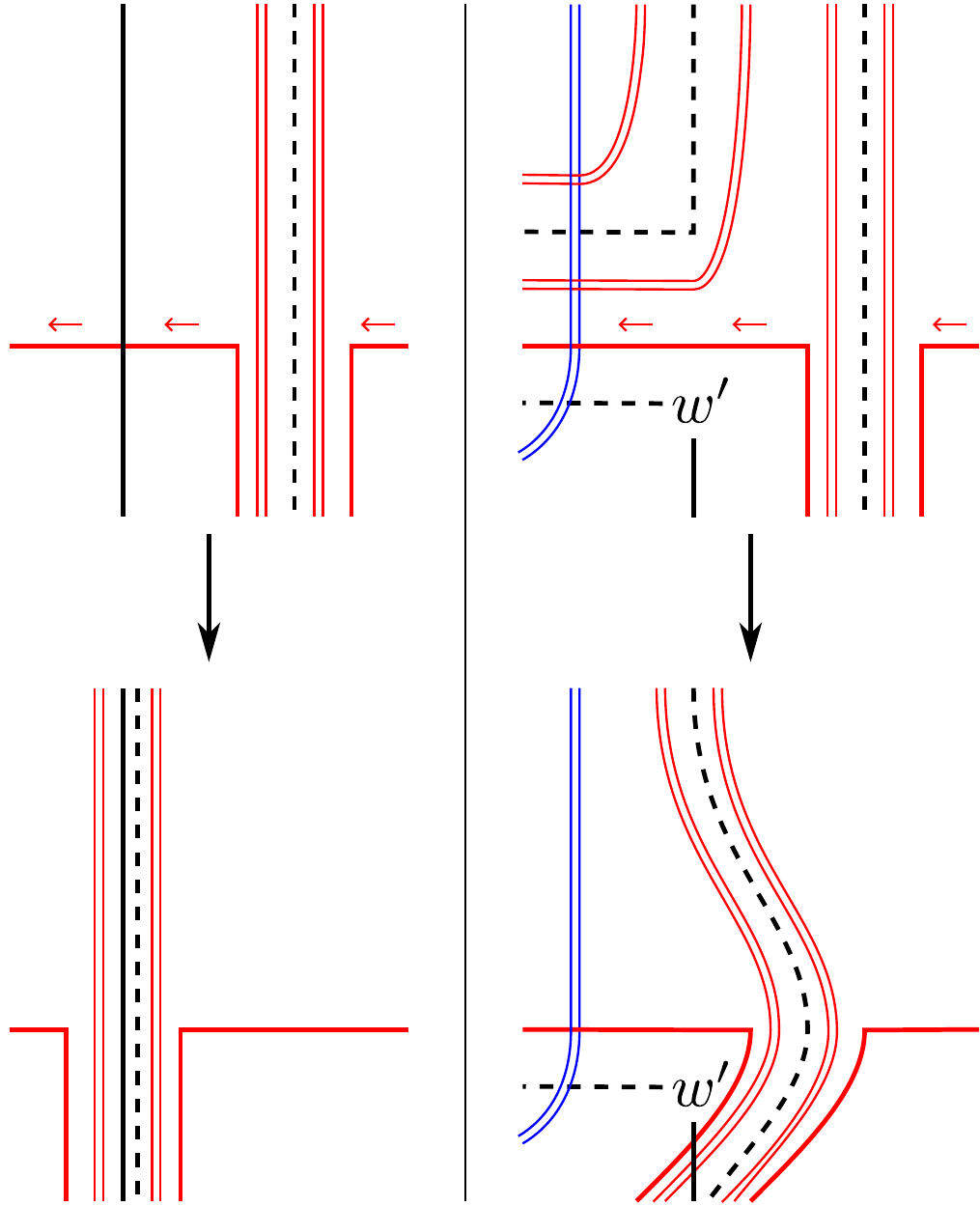}
\caption{During the move as shown in Figure \ref{pinch_diagram_4_new1}, the situation like the left happens in other places. Note that both curves cannot pass $w'$ as shown in the right one.}\label{pinch_diagram_4_new2}
\end{figure}

Finally, shrink the tongue-shaped part of $\alpha'$ to return $\alpha'$ to the standard position (the ``horizontal line"), and move $z'$, $k'_\alpha$ and $k'_\beta$ accordingly. 
See Figure \ref{pinch_diagram_5}. 
Note that since $k_\beta$ does not meet $\beta$, we can shrink $\alpha'$ without $z'$ touching $\beta$.
Then, recall the closest string of $k_\beta$ which is on the left of $z$ in the original diagram (before the pinch move). We see that the final position of $z'$ is on this string.

\begin{figure}
\centering
\includegraphics[scale=0.35]{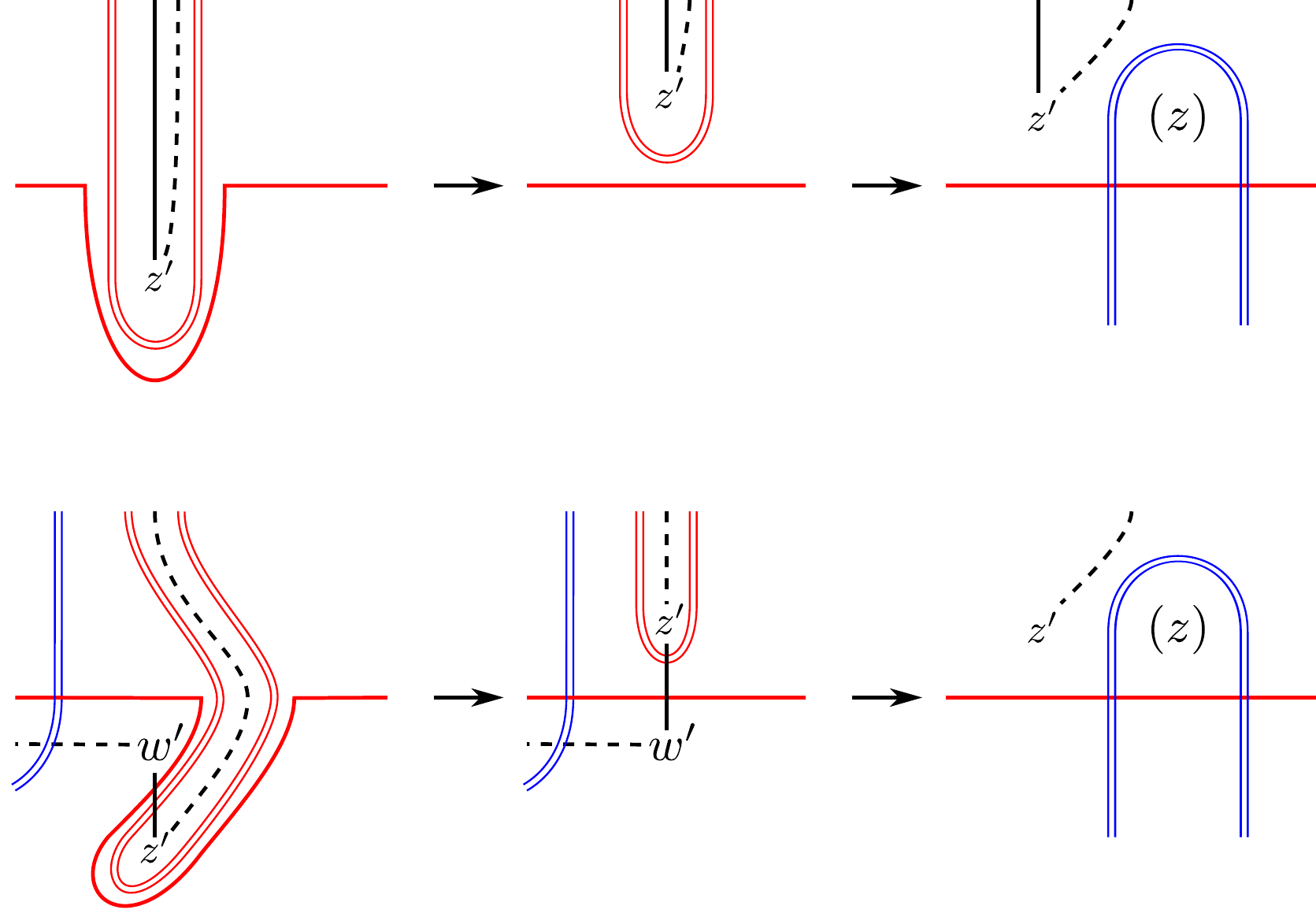}
\caption{Return $\alpha'$ and move $z'$, $k'_\alpha$ and $k'_\beta$ accordingly. 
(Top) The case where $z'$ does not pass by the side of $w'$. 
It stops when the situation on the right occurs. (Bottom) The case where $z'$ passes by the side of $w'$. }\label{pinch_diagram_5}
\end{figure}

As a result, the $(1,1)$--diagram looks like to have moved only the two points as shown in Figure \ref{pinch_diagram_6} by the pinch move. 
This movement of two points is also called the pinch move. 
Note that $\alpha$ and $\alpha'$ are equal as sets, and so are $\beta$ and $\beta'$. 

\begin{figure}
\centering
\includegraphics[scale=0.4]{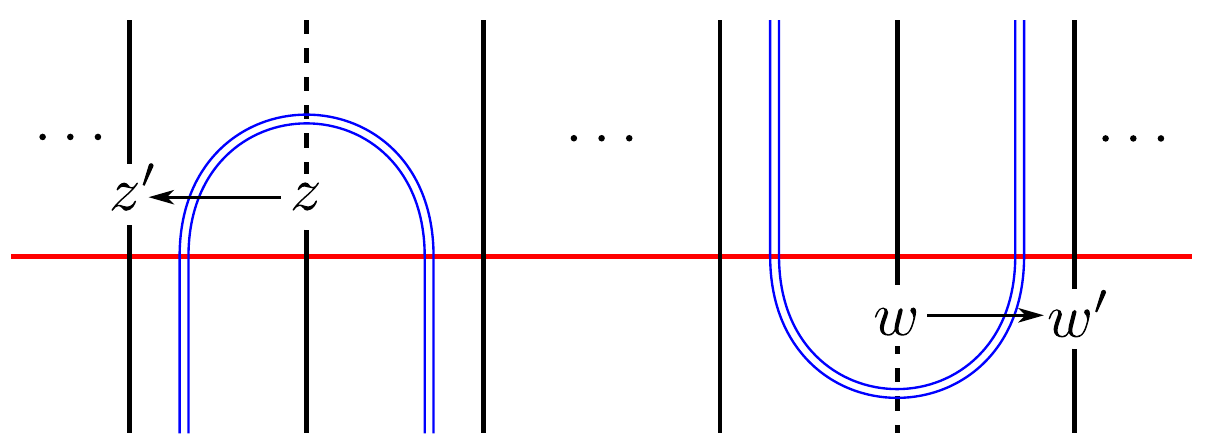}
\caption{The movement of two base points before and after the pinch move. 
This is also called the pinch move. 
Remark that the solid (resp. dashed) line represents $k_\beta$ (resp. $k_\alpha$) rather than $k'_\beta$ (resp. $k'_\alpha$). }\label{pinch_diagram_6}
\end{figure}

\begin{example}
We demonstrate the above process for $T_{5,7}$. 
The pinch move changes $T_{5,7}$ into the unknot $T_{1,1}$. 
See Figures \ref{ex_T57_1_2}, \ref{ex_T57_3_4} and \ref{ex_T57_5_6}. 
Figure \ref{ex_T57_5_6} (left) shows a (non-reduced) $(1,1)$--diagram of $T_{1,1}$. 
Here, $\alpha'$ and $\beta'$ are exactly equal to $\alpha$ and $\beta$, respectively.
However, the positions of the two base points $z'$ and $w'$ are different from those of $z$ and $w$. 
(Recall that the positions of the base points are sufficient to recover the knot.) 
In fact, their locations are specified at the original $(1,1)$--diagram of $T_{5,7}$ as shown in Figure \ref{ex_T57_5_6} (right). 
Thus we see that the pinch move as Figure \ref{pinch_diagram_6} happens.

\begin{figure}
\centering
\includegraphics[scale=0.3]{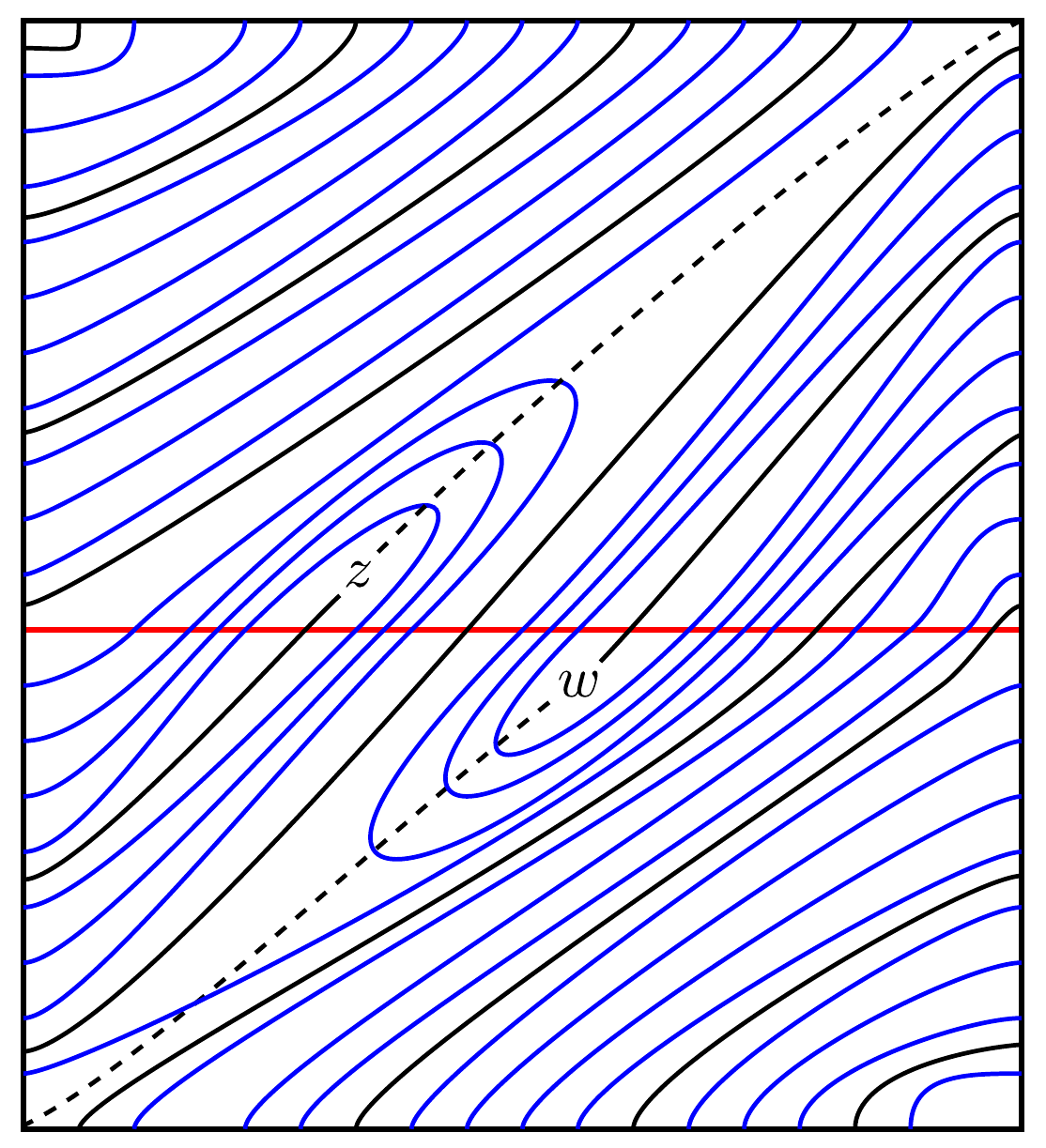}
\hspace{10mm}
\includegraphics[scale=0.3]{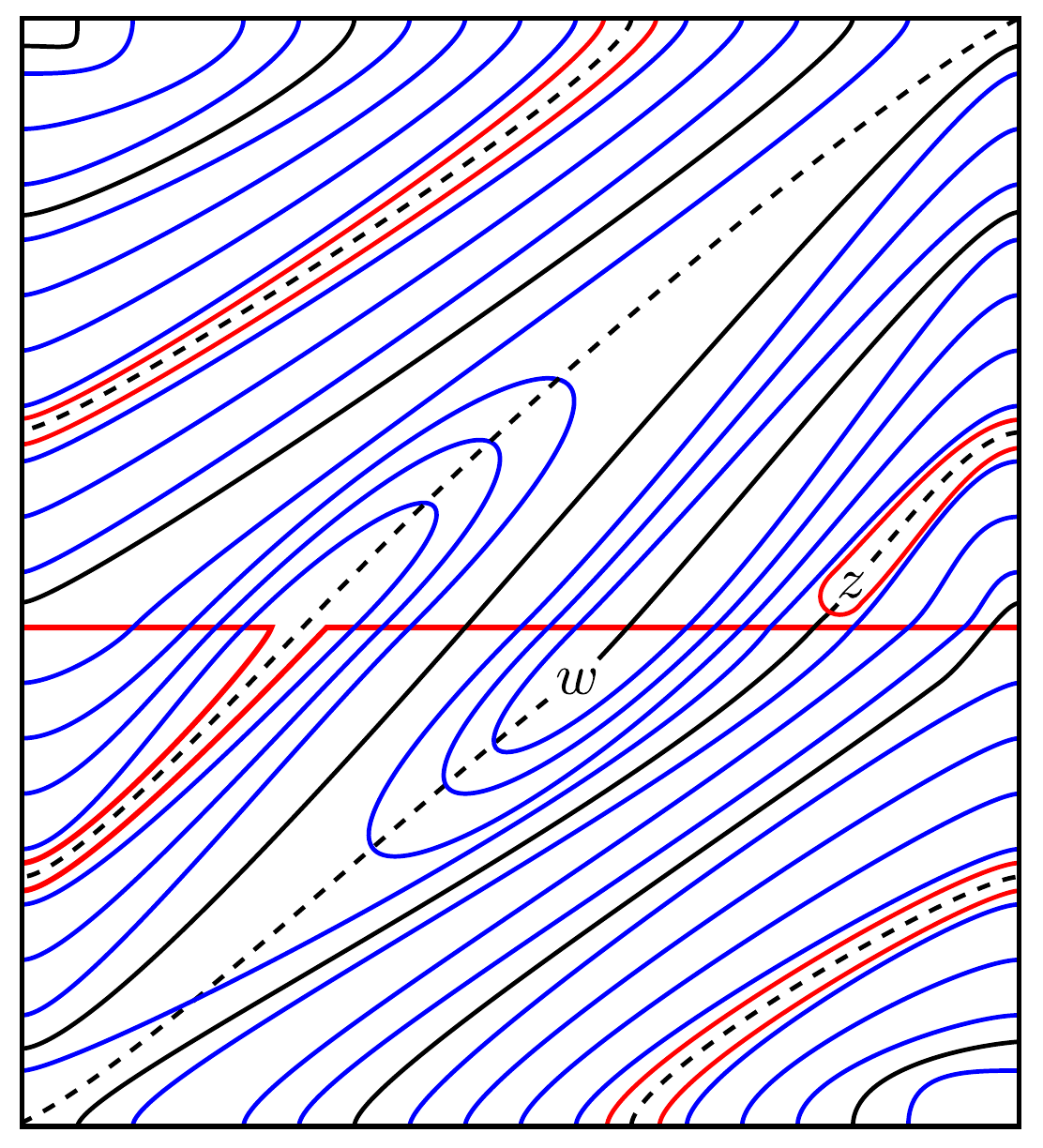}
\caption{(Left) The $(1,1)$--diagram of $T_{5,7}$. (Right) Move $z$ and $\alpha$.}\label{ex_T57_1_2}
\end{figure}

\begin{figure}
\centering
\includegraphics[scale=0.3]{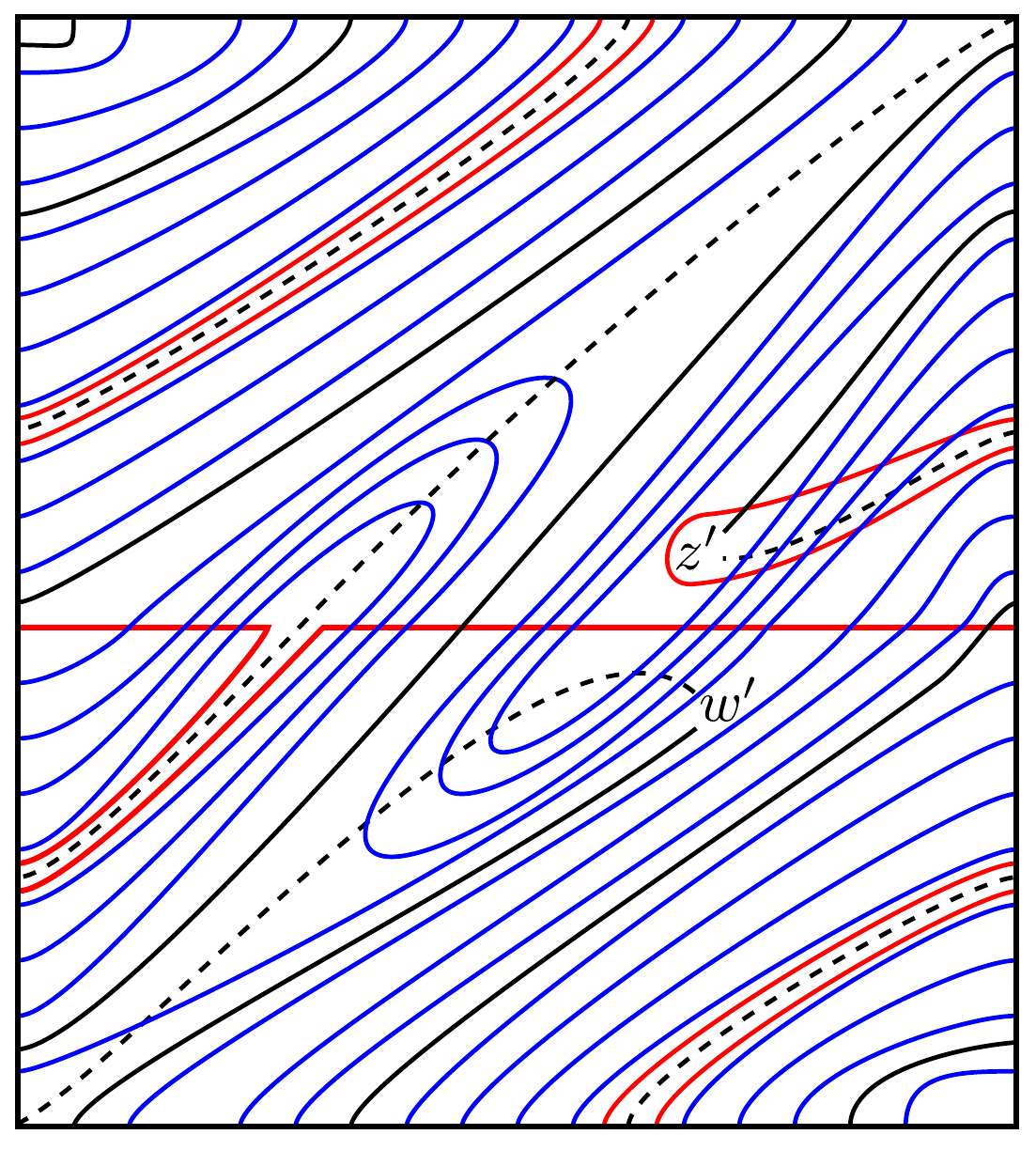}
\hspace{10mm}
\includegraphics[scale=0.3]{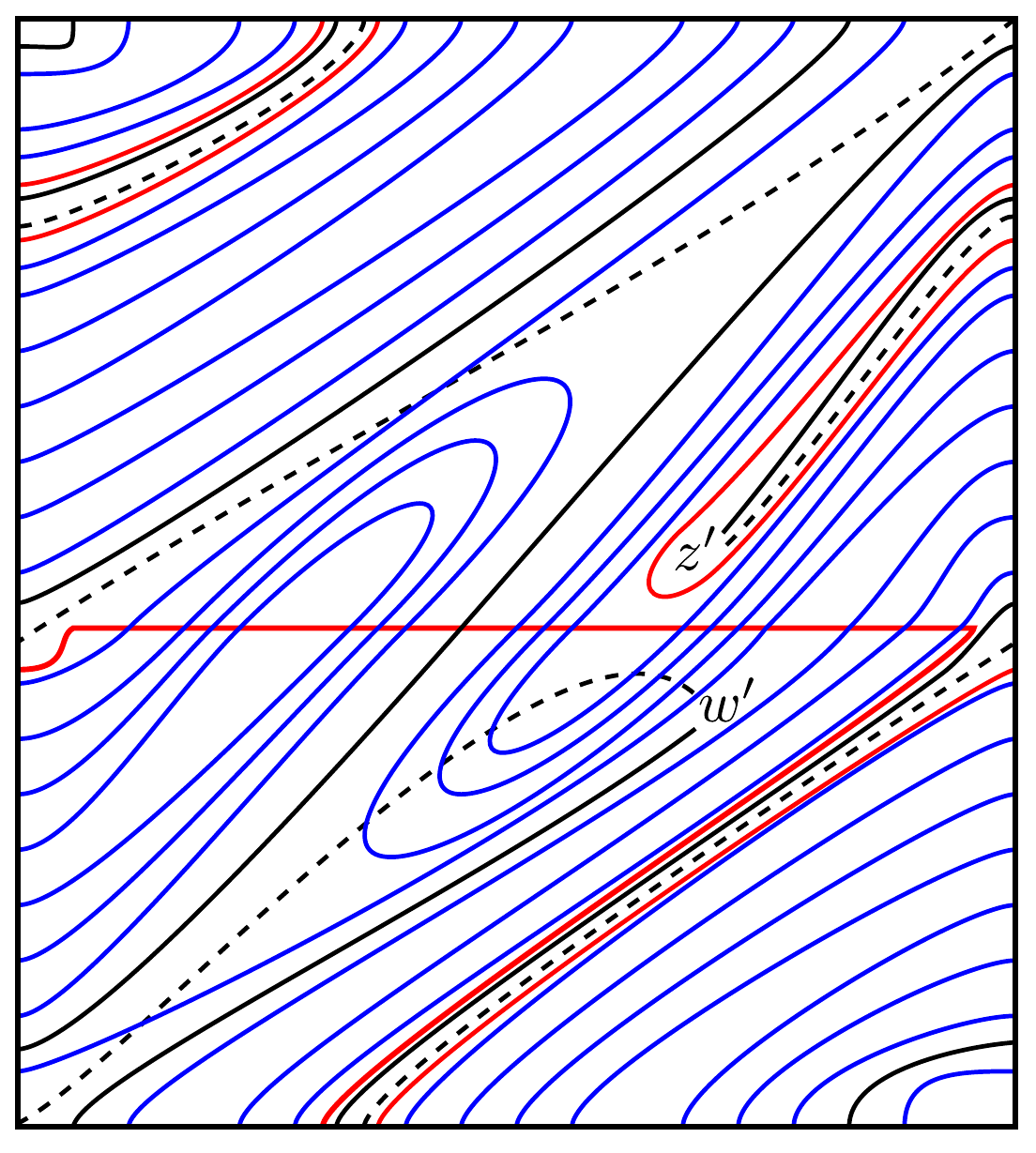}
\caption{(Left) Perform the pinch move. (Right) Shift $\alpha'$ together with $k'_\alpha$.}\label{ex_T57_3_4}
\end{figure}

\begin{figure}
\centering
\includegraphics[scale=0.3]{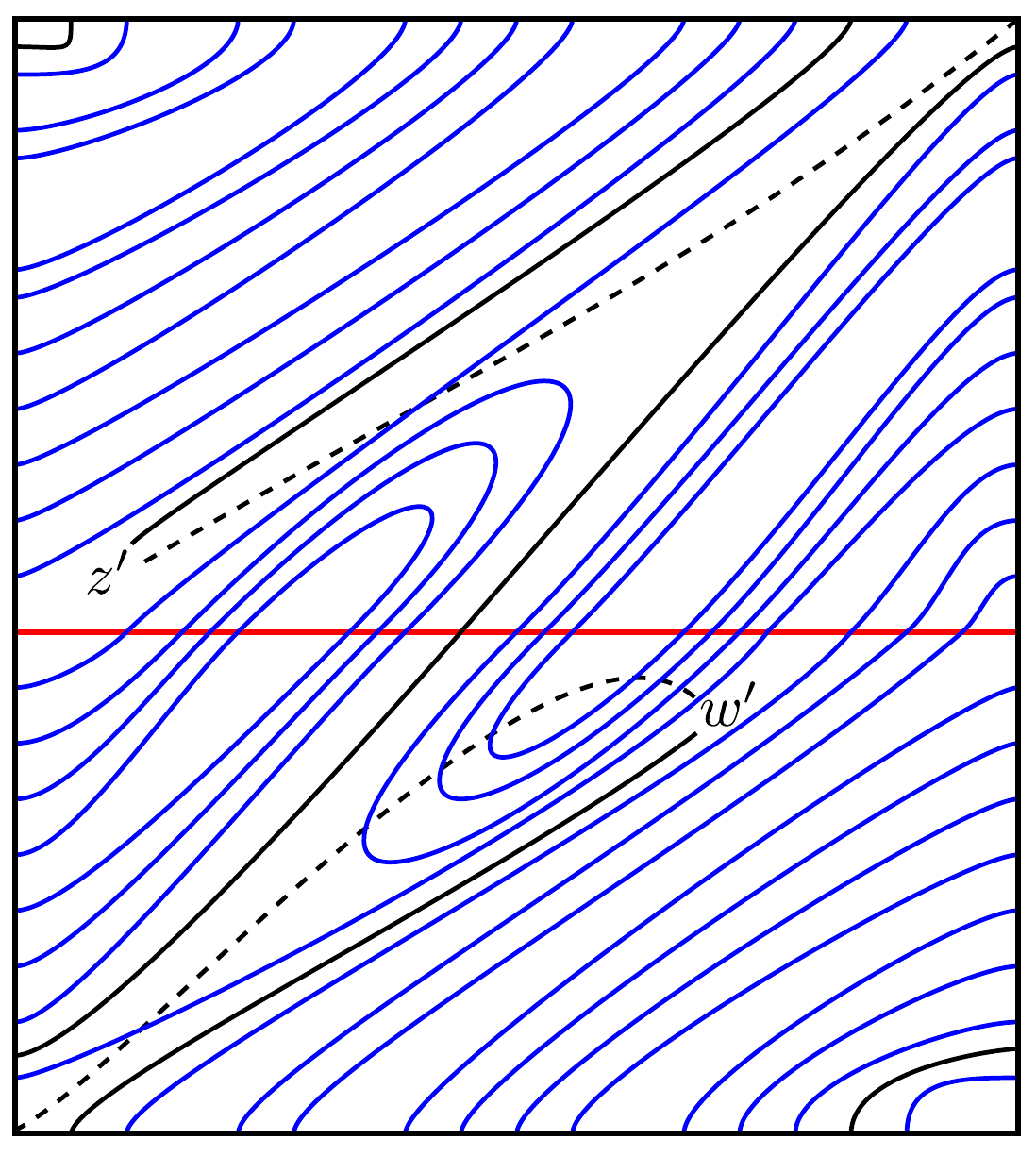}
\hspace{10mm}
\includegraphics[scale=0.3]{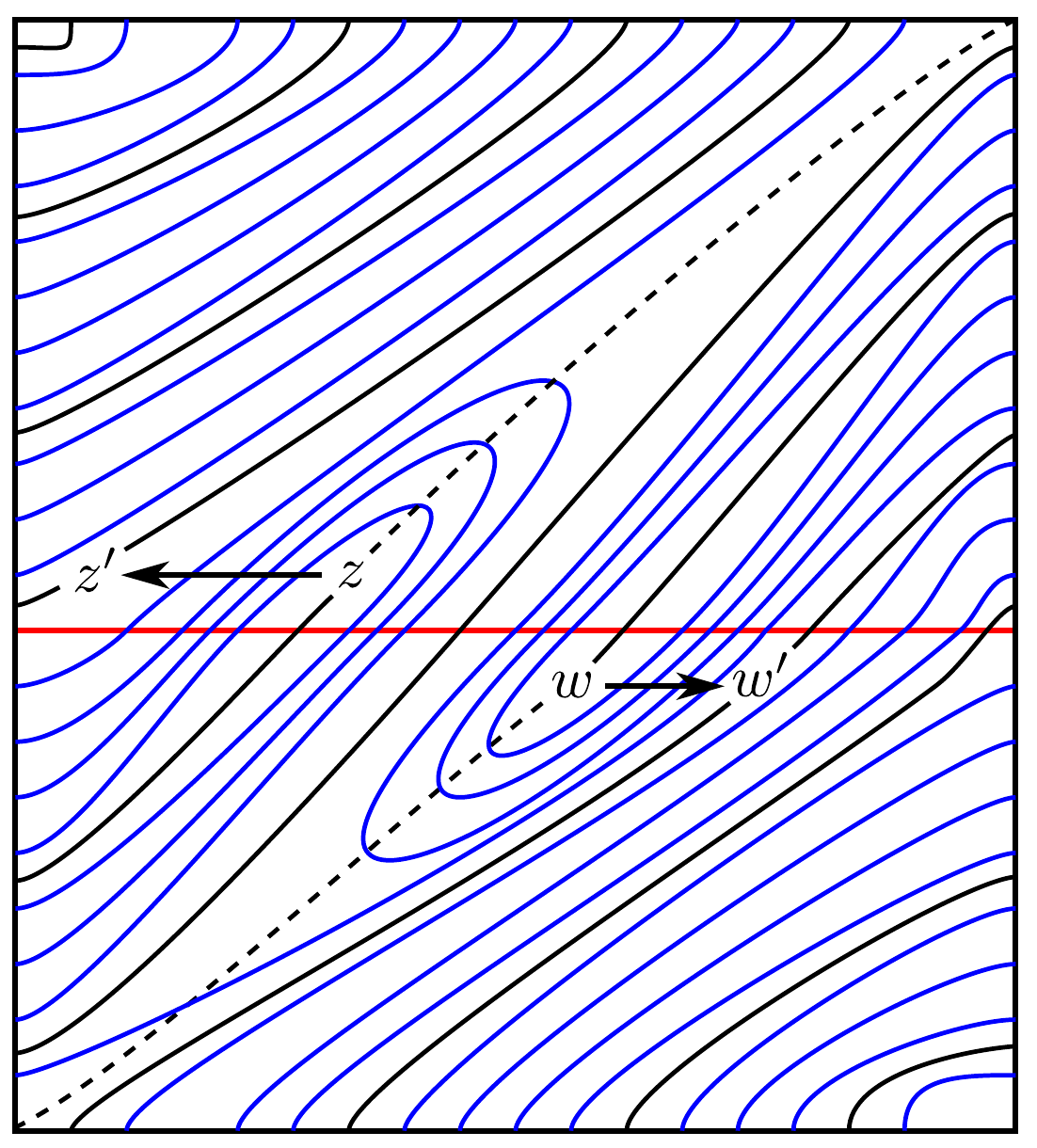}
\caption{(Left) Return $\alpha'$ to the standard position. This is a (non-reduced) $(1,1)$--diagram of $T_{1,1}$ after the pinch move. (Right) This shows the movement of two base points, called the pinch move. Remark that the dashed and solid lines represent $k_\alpha\cup k_\beta=T_{5,7}$.}\label{ex_T57_5_6}
\end{figure}

\end{example}


\section{Differential bigons in the universal cover of the $(1,1)$--diagram}
In this section, we give a behavior of $\Ord'$ before and after a pinch move, and prove Theorem \ref{thm_main}. Suppose that $T_{p',q'}$ is obtained by the pinch move on $T_{p,q}$.

\subsection{Points in differential bigons}
As mentioned above, to calculate $\CFK'(T_{p,q})$, it is helpful to consider the universal cover of the $(1,1)$--diagram $(\Sigma;\alpha,\beta;z,w)$ of the torus knot $T_{p,q}=k_\alpha\cup k_\beta$. 
Take the universal cover $\R^2$ so that lifts of $\alpha$ are horizontal lines. 
Let $\tilde\alpha$ and $\tilde\beta$ be one connected component of lifts of $\alpha$ and $\beta$, respectively. 
A bigon cobounded by $\tilde\alpha$ and $\tilde\beta$ contributing the differential on $\CFK'(K)$ is called a {\it differential bigon\/}. 
According to \cite{GLV18}, since a positive torus knot is an $L$--space knot, the intersection points between the curves $\tilde\beta$ and $\tilde\alpha$ occur in the opposite orders along $\tilde\alpha$ and $\tilde\beta$ (this condition is called {\it positive graphic\/} in \cite{GLV18}). 
Moreover, a differential bigon contains either $z$ or $w$, but not both. 
Hereafter, we consider only differential bigons containing $z$ points. 
The discussion that follows can be similarly applied to differential bigons containing $w$ points due to symmetry of the $(1,1)$--diagram. 
If there is no particular confusion, we denote the lifts of $z$, $w$, $k_\alpha$ and $k_\beta$, as well as the others lifts of $\alpha$ and $\beta$, simply as $z$, $w$, $k_\alpha$, $k_\beta$, $\alpha$ and $\beta$, respectively. 

Here, we observe a specific example to see how the number of $z,\ w,\ z'$ and $w'$ changes within differential bigons under the pinch move. 
Note that since $\alpha=\alpha'$ and $\beta=\beta'$ as sets, there is one-to-one correspondence between two sets of differential bigons of $(\Sigma;\alpha,\beta;z,w)$ and $(\Sigma;\alpha',\beta';z',w')$.

\begin{example}
We consider a certain differential disk in the universal cover of the $(1,1)$--diagram of $T_{5,7}$. 
See Figure \ref{ex_T57_8}. 
Let $\phi$ be the differential bigon indicated by shaded area. 
Then, $n_z(\phi)+n_w(\phi)=4+0=4$. 
By the pinch move, two points $z$ and $w$ move as shown there, and hence a disk $\phi'$ corresponding to $\phi$ satisfies $n_{z'}(\phi')+n_{w'}(\phi')=1+2=3$. 
Therefore, $n_{z'}(\phi')+n_{w'}(\phi')=n_z(\phi)+n_w(\phi)-1$ holds.

\begin{figure}
\centering
\includegraphics[scale=0.3]{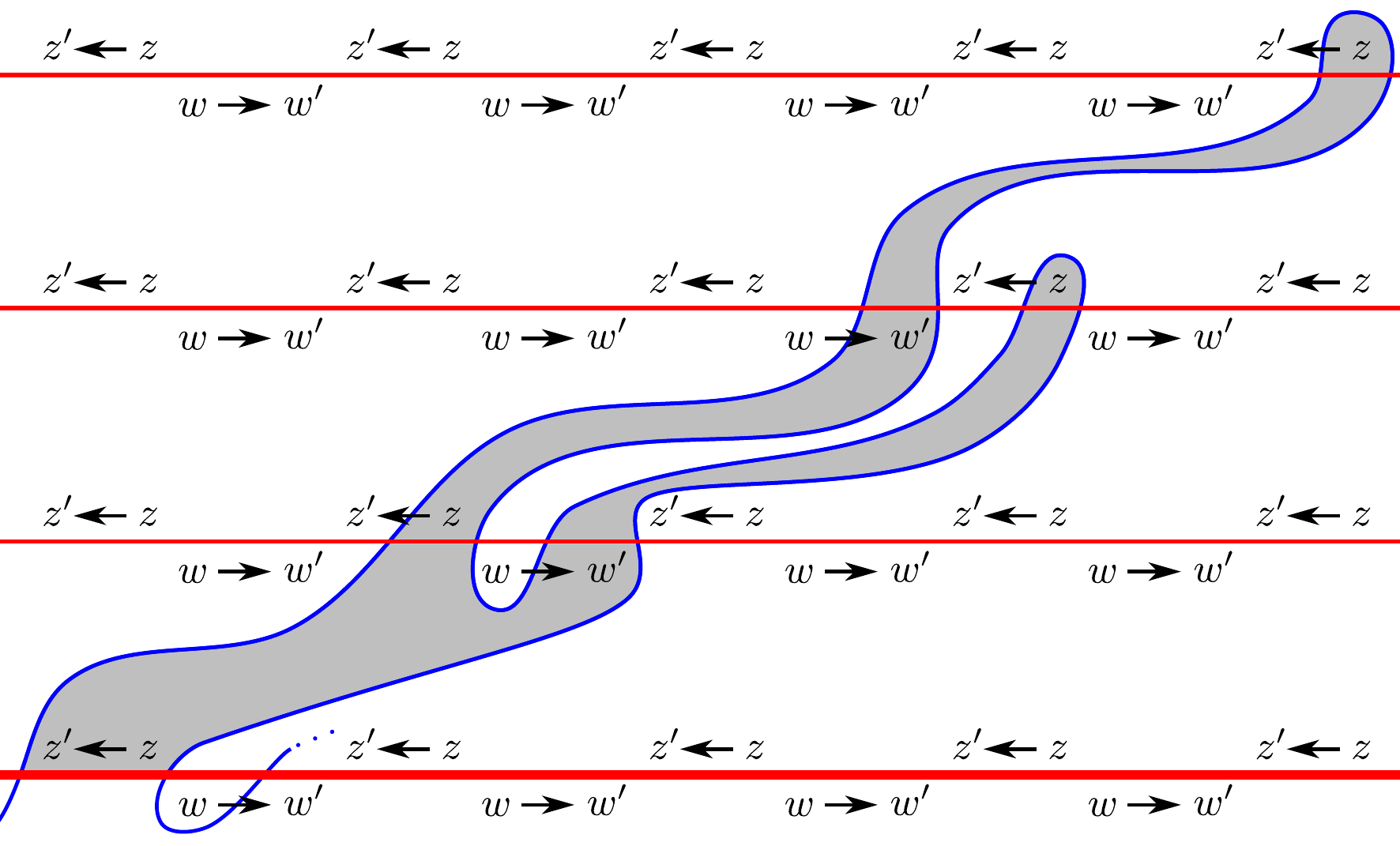}
\caption{A part of the universal cover of the $(1,1)$--diagram of $T_{5,7}$. 
The bold red line represents $\tilde\alpha$, other red lines represent lifts of $\alpha$ that is not $\tilde\alpha$. 
$\phi$ is the shaded bigon.}\label{ex_T57_8}
\end{figure}

\end{example}

In fact, the pinch move always lowers the number of points in a differential bigon by one (Proposition \ref{lem_bigon_point}). 
We show this below. 

By cutting a differential bigon along lifts of $\alpha$, we obtain the connected components, called {\it domains\/}. 
The boundary of a domain consists of subarcs of $\alpha$, $\tilde\alpha$ and $\tilde\beta$. 
Remark that a domain differs from a region used in Subsection \ref{Regions}. 
In particular, a region does not contain $\beta$--arcs in its interior, but a domain may contain $\beta$--arcs in its interior. 
There are three type of domains; {\it end--domain\/}, {\it rectangle--domain\/} and {\it $Y_n$--domain\/}, as illustrated in Figures \ref{domain_1} and \ref{domain_2}. 
A rectangle--domain can be regarded as a $Y_0$--domain, but for the sake of clarity, we deliberately describe them separately. 
Note that for a domain, each subarc of $\alpha$ or $\tilde\alpha$ on the boundary meets $k_\beta$. 
For a domain $D$, let $n_\circ (D)$ be the number of $\circ$ points in $D$ where $\circ=z,w,z',w'$.

\begin{figure}
\centering
\includegraphics[scale=0.45]{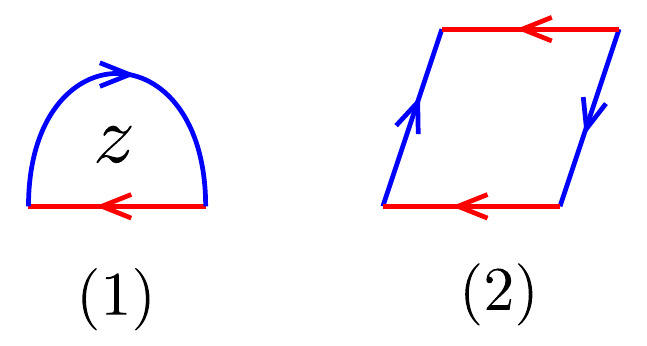}
\caption{Red lines are lifts of $\alpha$ and blue lines are $\tilde\beta$. 
(1) An {\it end--domain\/}. This contains exactly one $z$ point. 
(2) A {\it rectangle--domain\/}. This contains at most two $z$ points (see the proof of Lemma \ref{lem_domain_rectangle}).}\label{domain_1}
\end{figure}

\begin{figure}
\centering
\includegraphics[scale=0.3]{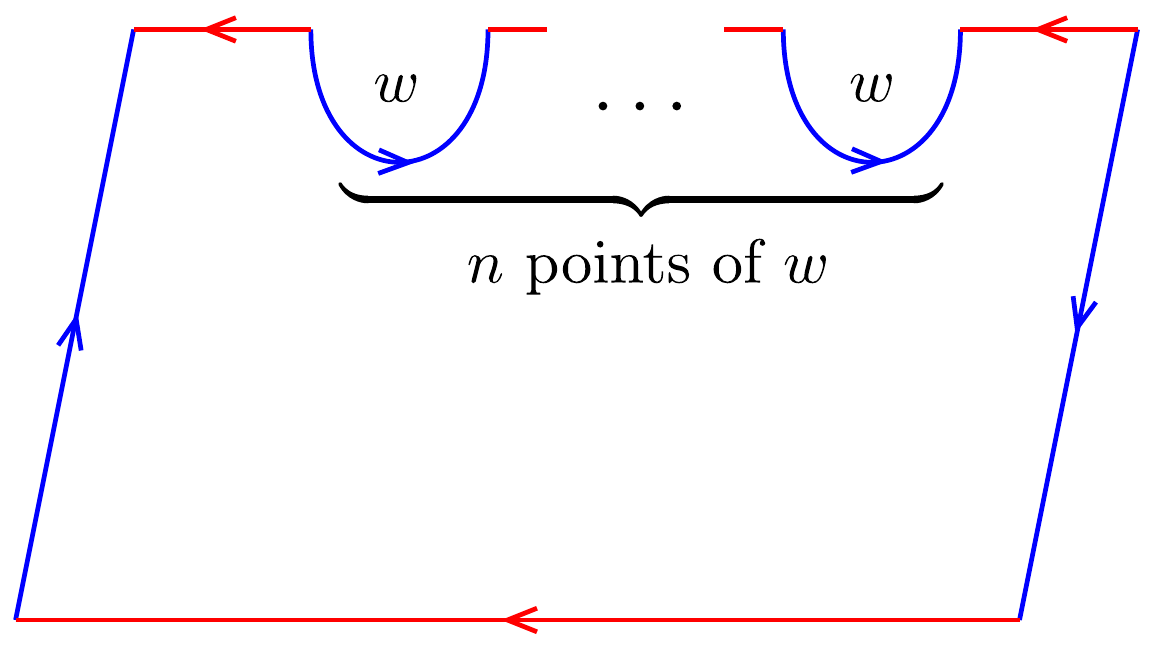}
\caption{A {\it $Y_n$--domain\/}. This contains $n$ to $n+2$ $z$ points (see the proof of Lemma \ref{lem_domain_rectangle}). 
A rectangle--domain can be regarded as a $Y_0$--domain.}\label{domain_2}
\end{figure}

Here, we provide lemmas regarding how points in domains change before and after the pinch move. 
From now on, bold blue lines in following figures represent $\tilde\beta$, other blue lines represent lifts of $\beta$ that is not $\tilde\beta$.

\begin{lemma}\label{lem_domain_rectangle}
For a rectangle--domain $D$ and its corresponding domain $D'$ by the pinch move, $n_{z'}(D')+n_{w'}(D')=n_z(D)+n_w(D)$. 
\end{lemma}
\begin{proof}
$D$ contains at most two $z$ points. 
Otherwise, some $k_{\beta}$ must intersect $k_{\alpha}$ (see Figure \ref{fig_lem_domain_rectangle_z}), which is a contradiction. 
Thus, we consider three cases according to the number of $z$ contained in $D$.

\begin{figure}
\centering
\includegraphics[scale=0.3]{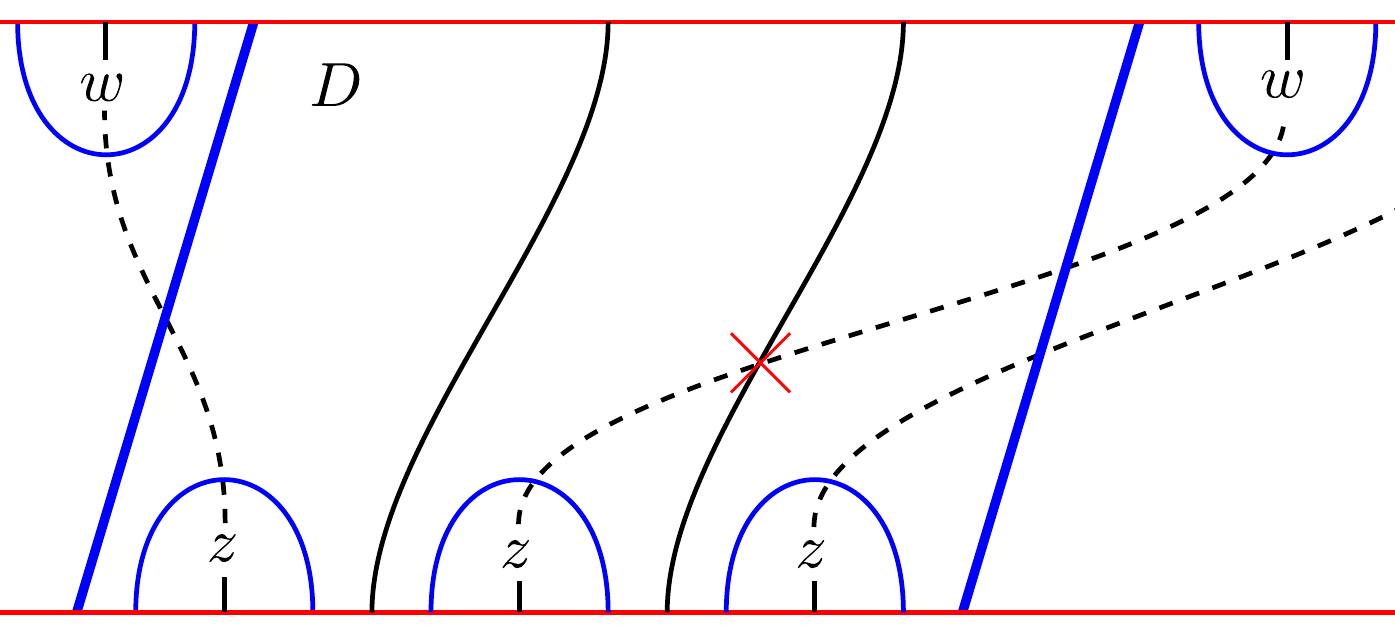}
\caption{A rectangle--domain $D$ containing three $z$ points. 
Recall that there is no $w$ point in $D$. 
The rectangle cobounded by the two red lines $\alpha$ and the bold blue lines $\tilde\beta$ is $D$. 
For example, if $k_\alpha$ is located in the position shown there, $k_\beta$ intersects $k_\alpha$ (this $k_{\beta}$ exists since $p\ge 2$). }\label{fig_lem_domain_rectangle_z}
\end{figure}

\begin{itemize}
\item[{\bf Case 1:}] $D$ contains no $z$ point. 
See Figure \ref{fig_lem_domain_rectangle_z0}. 
The position of $\tilde\beta$ implies that there exist segments of $\beta$ indicated by $*$. 
Moreover, by applying Lemma \ref{lem_Y} or \ref{lem_H} to the $A$--, $Y$-- or $H$--region in both side of $D$, two segments of $k_{\beta}$ indicated by $**$ exist.
Hence, $z'$ and $w'$ do not come into $D$ when we perform the pinch move. 
Therefore, we have $n_{z'}(D)+n_{w'}(D)=n_z(D)+n_w(D)=0$.

\begin{figure}
\centering
\includegraphics[scale=0.35]{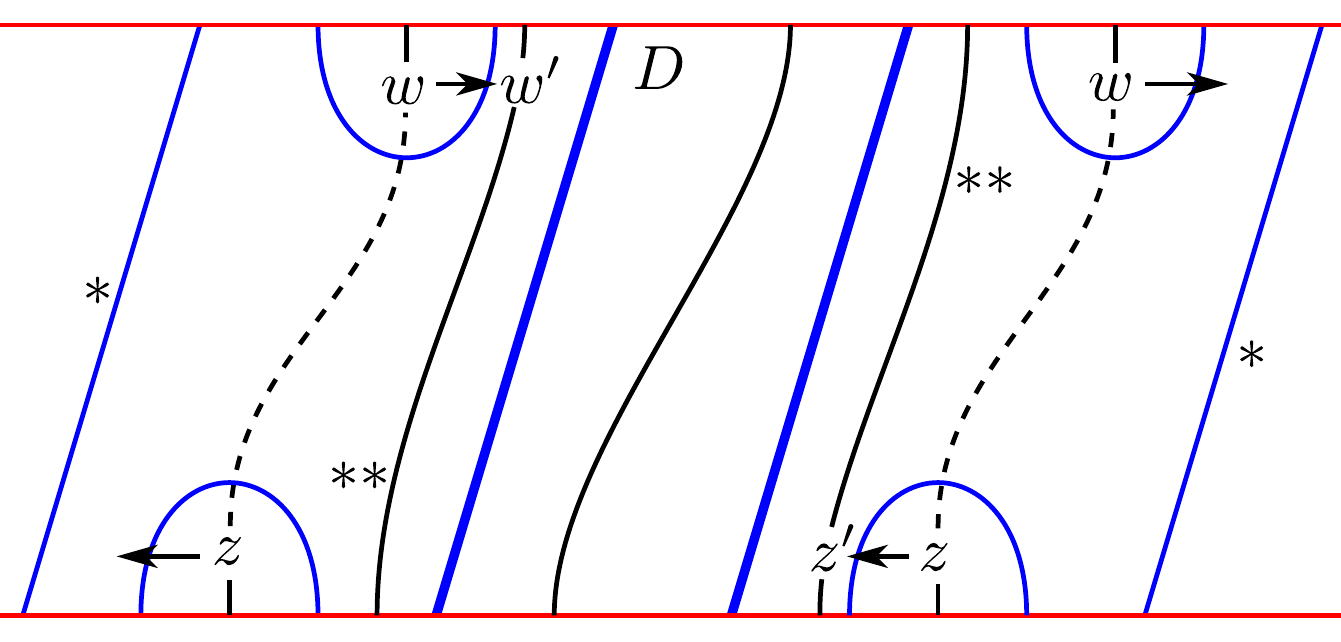}
\caption{The case of a rectangle--domain $D$ containing no $z$.}\label{fig_lem_domain_rectangle_z0}
\end{figure}

\item[{\bf Case 2:}] $D$ contains one $z$ point. 
There are two patterns; the case where $z$ goes outside of $D$ by the pinch move (Figure \ref{fig_lem_domain_rectangle_z1_2}), or not (Figure \ref{fig_lem_domain_rectangle_z1_1}). 
The discussions of both of cases are similar to each other. 
See Figures \ref{fig_lem_domain_rectangle_z1_2} and \ref{fig_lem_domain_rectangle_z1_1}. 
First, the position of $k_\beta$ is determined as shown in $(1)$ depending on whether $z$ goes outside of $D$ or not. 
Furthermore, this in turn determines the position of $k_\alpha$ as illustrated there. 
Second, there are segments of $\beta$ indicated by $(2)$, because of the position of $\tilde\beta$. 
Then, Lemma \ref{lem_Y} implies that there are segments of $k_\beta$ as in $(3)$. 
Therefore, we have $n_{z'}(D')+n_{w'}(D')=n_z(D)+n_w(D)=1$ for both cases.

\begin{figure}
\centering
\includegraphics[scale=0.35]{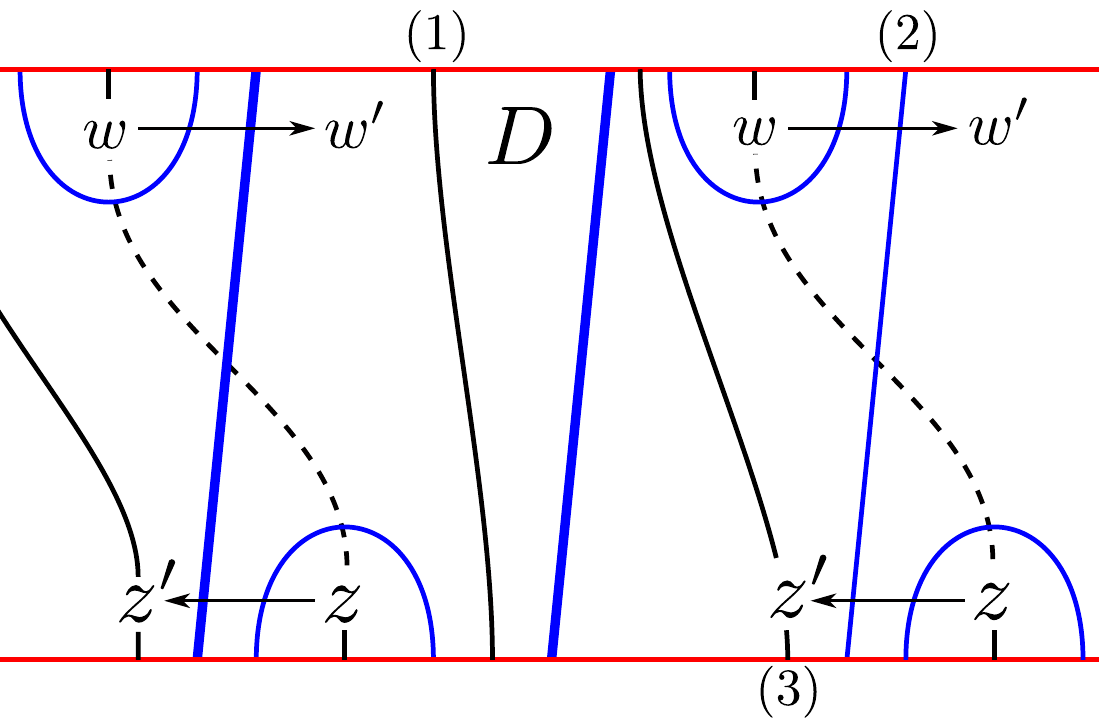}
\caption{The case of a rectangle--domain $D$ containing one $z$ which goes outside of $D$ by the pinch move. }\label{fig_lem_domain_rectangle_z1_2}
\end{figure}

\begin{figure}
\centering
\includegraphics[scale=0.35]{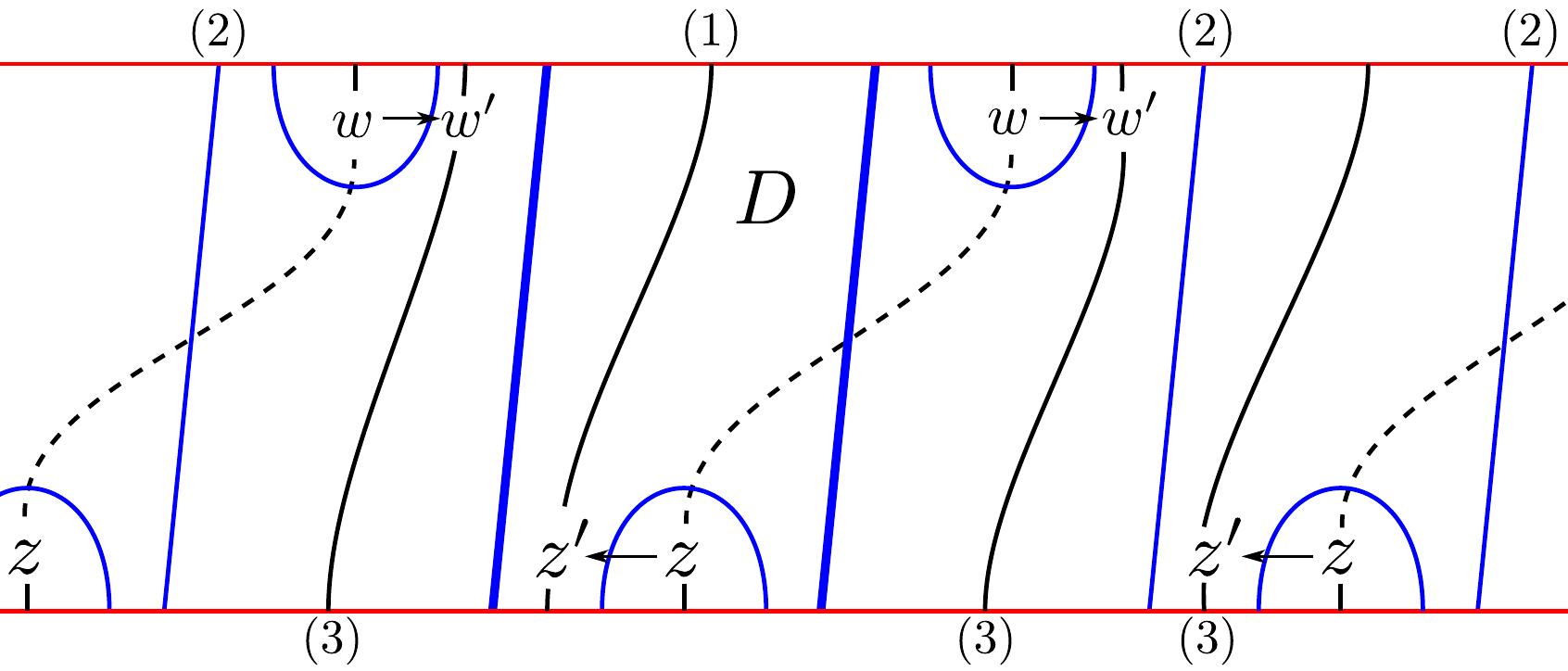}
\caption{The case of a rectangle--domain $D$ containing one $z$ which does not go outside of $D$ by the pinch move. }\label{fig_lem_domain_rectangle_z1_1}
\end{figure}

\item[{\bf Case 3:}] $D$ contains two $z$ points. The argument of this case is the same as in the case of $D$ with one $z$ point. Note that the right one of $z$ within $D$ cannot go outside of $D$. 
\end{itemize}
\end{proof}

\begin{lemma}\label{lem_domain_Y}
For a $Y_n$--domain $D$ and its corresponding domain $D'$ by the pinch move, $n_{z'}(D')+n_{w'}(D')=n_z(D)+n_w(D)+n$. 
\end{lemma}
\begin{proof}
A $Y_n$--domain contains only $n$ to $n+2$ points of $z$: 
The reason why a $Y_n$--domain can contain at most $n+2$ points is the same as the reason why a rectangle--domain can contain at most two points. 
If a $Y_n$--domain contains less than $n$ $z$ points, then at least one of $w$, which is located just above the $Y_n$--domain, is connected to $z$ by $k_\alpha$, which is outside of the domain. 
Then, there would be $k_\beta$ meeting $k_\alpha$ (see Figure \ref{fig_lem_domain_Y}).

Also, as in Figure \ref{fig_lem_domain_Y}, $n$ points of $w$ go inside a $Y_n$--domain by the pinch move. 
The fact that the number of other points within the domain does not change can be shown in a manner similar to Lemma \ref{lem_domain_rectangle}, so we have $n_{z'}(D')+n_{w'}(D')=n_z(D)+n_w(D)+n$.
\end{proof}

\begin{figure}
\centering
\includegraphics[scale=0.35]{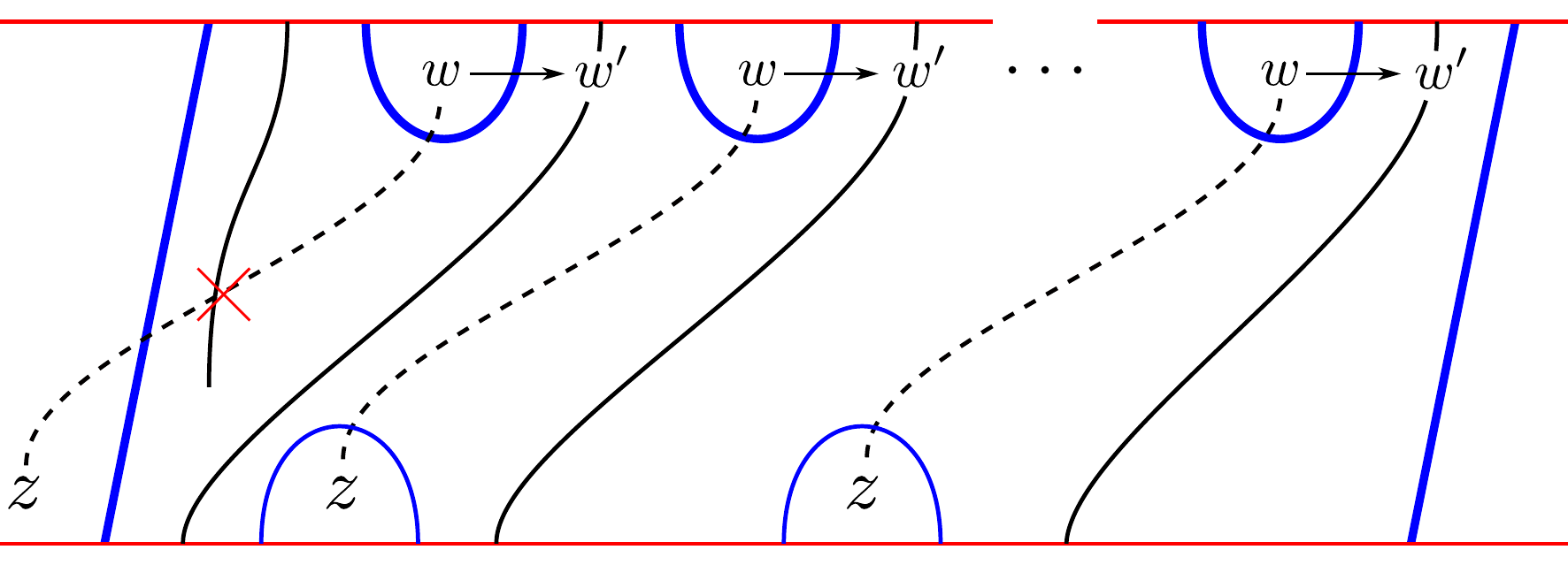}
\caption{If a $Y_n$--domain contains less than $n$ $z$ points, then there is $w$ connecting to an outside $z$ by $k_\alpha$ as in the figure. Then, $k_\beta$ and $k_\alpha$ intersect. Moreover, $n$ points of $w$ go inside the $Y_n$--domain by the pinch move.}\label{fig_lem_domain_Y}
\end{figure}

\begin{lemma}\label{lem_bigon_point}
Let $\phi$ be a differential bigon for $T_{p,q}$, and $\phi'$ the differential bigon for $T_{p',q'}$ corresponding to $\phi$. 
Then, $n_{z'}(\phi')+n_{w'}(\phi')=n_z(\phi)+n_w(\phi)-1$.
\end{lemma}
\begin{proof}
Assume that there are $Y_{n_1}$--, $Y_{n_2}$--,\ldots, $Y_{n_k}$--domains where $n_1,\ldots,n_k$ are positive integers when we cut $\phi$ along all lifts of $\alpha$ (if there is no $Y_n$--domain, set $k=1$ and $n_1=0$). 
Then, there are $n_1+\cdots+n_k+1$ end--domains of $\phi$. 
By Lemmas \ref{lem_domain_rectangle}, \ref{lem_domain_Y} and the fact that $z$ in an end-domain goes out by the pinch move, we have 
\begin{align*}
n_{z'}(\phi')+n_{w'}(\phi')&=n_z(\phi)+n_w(\phi)+n_1+\cdots+n_k-(n_1+\cdots+n_k+1)\\
&=n_z(\phi)+n_w(\phi)-1.
\end{align*}
\end{proof}

\begin{remark}
We do not know whether a $Y_n$--domain with $n\ge 2$ actually appears or not. 
However, as we have discussed so far, either option is acceptable. 
\end{remark}

\subsection{Proof of Theorem \ref{thm_main}}
Since the positions $\tilde\alpha$ and $\tilde\beta$ are unchanged before and after the pinch move, $\CFK'(T_{p',q'})$ is equal to $\CFK'(T_{p,q})$ as an $\F_2[U]$--module. 

\begin{lemma}
The set of cycles in $\CFK'(T_{p',q'})$ coincides with one of $\CFK'(T_{p,q})$.
\end{lemma}
\begin{proof}
By Lemma \ref{lem_bigon_point}, the differential of $\CFK'(T_{p',q'})$ differs from one of $\CFK'(T_{p,q})$ only by the action of $U$. 
Since the complexes are free $\F_2[U]$--modules, a cycle in $\CFK'(T_{p',q'})$ is also a cycle in $\CFK'(T_{p,q})$, and vice versa.
\end{proof}
We denote $c'\in\CFK'(T_{p',q'})$ as the corresponding cycle to a cycle $c\in\CFK'(T_{p,q})$.

\begin{proposition}\label{prop_unorihomology_pinch}
$\HFK'(T_{p',q'})\cong U\cdot \HFK'(T_{p,q})$ as an $\F_2[U]$--module. In particular, $\Ord'(T_{p',q'})=\Ord'(T_{p,q})-1$.
\end{proposition}
\begin{proof}
Let $f\colon\HFK'(T_{p',q'})\to U\cdot \HFK'(T_{p,q})$ be an $\F_2[U]$--module map defined by $f([c'])=U\cdot [c]$ for a cycle $c'$ in $\CFK'(T_{p',q'})$. 
We can show that it is a well-defined map: 
If $[c']=[d']$ in $\HFK'(T_{p',q'})$, there is a chain $b'\in\CFK'(T_{p',q'})$ such that $\partial(b')=c'+d'$. 
Lemma \ref{lem_bigon_point} implies $\partial(b)=U\cdot(c+d)$, so $U\cdot [c]=U\cdot [d]$ in $U\cdot \HFK'(T_{p,q})$. 
Therefore, $f$ is well-defined.

Also, an $\F_2[U]$--module map $g\colon U\cdot \HFK'(T_{p,q})\to \HFK'(T_{p',q'})$ defined by $g(U\cdot [c])=[c']$ for a cycle $c$ in $\CFK'(T_{p,q})$ is the well-defined inverse map of $f$, so $\HFK'(T_{p',q'})\cong U\cdot \HFK'(T_{p,q})$ as an $\F_2[U]$--module.
\end{proof}

\begin{proposition}\label{prop_Ord_P}
For any positive torus knot $T_{p,q}$, $\Ord'(T_{p,q})=P(p,q)$.
\end{proposition}
\begin{proof}
We prove it by an induction on the pinch number. 
Note that a torus knot with $P(p,q)=0$ is the unknot $O$, which satisfies $\Ord'(O)=0$. 

Assume that a torus knot $T_{p,q}$ with $P(p,q)=n$ has $\Ord'(T_{p,q})=n$. 
Consider a torus knot $T_{p,q}$ with $P(p,q)=n+1$. 
If $T_{p',q'}$ is obtained by the pinch move on $T_{p,q}$, $P(p',q')=n$, and hence $\Ord'(T_{p',q'})=n$ by the assumption. 
Thus, Lemma \ref{prop_unorihomology_pinch} implies $\Ord'(T_{p,q})=n+1$, and we are done.
\end{proof}

\begin{proof}[Proof of Theorem \ref{thm_main}]
By Proposition \ref{prop_GM} and the definition, we obtain $\Ord'(T_{p,q})\le \u^u_b(T_{p,q})\le P(p,q)$. 
Proposition \ref{prop_Ord_P} implies $\u^u_b(T_{p,q})=P(p,q)$. 
Hence, we have the conclusion by Proposition \ref{prop_pinch_formula}.
\end{proof}



\begin{thebibliography}{AA}

\bibitem{AK14}
Tetsuya Abe and Taizo Kanenobu,
{\it Unoriented band surgery on knots and links},
Kobe J. Math. {\bf 31} (2014), no. 1--2, 21--44.

\bibitem{Bat14}
Joshua Batson,
{\it Nonorientable slice genus can be arbitrarily large},
Math. Res. Lett. {\bf 21} (2014), no. 3, 423--436.

\bibitem{BE19}
Zac Bettersworth and Claus Ernst,
{\it Incoherent nullification of torus knots and links},
J. Knot Theory Ramifications {\bf 28} (2019), no. 5, 1950033, 23 pp.

\bibitem{Fan19}
Haofei Fan,
{\it Unoriented cobordism maps on link Floer homology},
[PhD Thesis, University of California, Los Angeles], 2019.

\bibitem{GMM05}
Hiroshi Goda, Hiroshi Matsuda, and Takayuki Morifuji,
{\it Knot Floer homology of $(1,1)$--knots},
Geom. Dedicata {\bf 112} (2005), 197--214.

\bibitem{GM23}
Sherry Gong and Marco Marengon,
{\it Nonorientable link cobordisms and torsion order in Floer homologies},
Algebr. Geom. Topol. {\bf 23} (2023), no.6, 2627--2672.

\bibitem{GLV18}
Joshua Evan Greene, Sam Lewallen and Faramarz Vafaee,
{\it $(1,1)$ $L$--space knots},
Compos. Math. {\bf 154} (2018), no. 5, 918--933.

\bibitem{Hom17}
Jennifer Hom, 
{\it A survey on Heegaard Floer homology and concordance}, 
J. Knot Theory Ramifications {\bf 26} (2017), no.2, 1740015, 24 pp.

\bibitem{HNT90}
Jim Hoste, Yasutaka Nakanishi and Kouki Taniyama,
{\it Unknotting operations involving trivial tangles},
Osaka J. Math.  {\bf 27} (1990), no. 3, 555--566.

\bibitem{JV21}
Stanislav Jabuka and Cornelia Van Cott,
{\it On a nonorientable analogue of the Milnor conjecture},
Algebr. Geom. Topol. {\bf 21} (2021), no. 5, 2571--2625.

\bibitem{KM09}
Taizo Kanenobu and Yasuyuki Miyazawa,
{\it $H(2)$-unknotting number of a knot},
Commun. Math. Res. {\bf 25} (2009), no. 5, 433--460.

\bibitem{Man16}
Ciprian Manolescu, 
{\it An introduction to knot Floer homology}, 
Physics and mathematics of link homology, Contemp. Math., {\bf 680} Centre Rech. Math. Proc. American Mathematical Society, Providence, RI, 2016, 99--135.

\bibitem{OSS17A} 
Peter Ozsv\'ath, Andr\'as Stipsicz and Zolt\'an Szab\'o, 
{\it Concordance homomorphisms from knot Floer homology}, 
Adv. Math. {\bf 315} (2017), 366--426.

\bibitem{OSS17B}
Peter Ozsv\'ath, Andr\'as Stipsicz and Zolt\'an Szab\'o, 
{\it Unoriented knot Floer homology and the unoriented four-ball genus},
Int. Math. Res. Not. IMRN 2017, no. {\bf 17}, 5137--5181.

\bibitem{OS04}
Peter Ozsv\'ath and Zolt\'an Szab\'o,
{\it Holomorphic disks and knot invariants}, 
Adv. Math. {\bf 186} (2004), no.1, 58--116.

\bibitem{OS05}
Peter Ozsv\'ath and Zolt\'an Szab\'o, 
{\it On knot Floer homology and lens space surgeries}, 
Topology {\bf 44} (2005), no.6, 1281--1300.


\end{thebibliography}
\end{document}